\documentclass[10pt]{article}

\usepackage[utf8]{inputenc}
\usepackage{enumerate}

\usepackage{amsthm}

\usepackage{amsmath}
\usepackage{amsfonts}
\usepackage{amssymb}
\usepackage{amscd}
\usepackage{latexsym}

\def\C{\mathbb{C}}

\def\N{\mathbb{N}}
\def\R{\mathbb{R}}
\def\RR{\mathbb R^d}

\def\D{{\mathcal D}}
\def\F{\mathcal F}
\def\H{\mathcal H}
\def\K{\mathcal K}

\def\S{\mathcal S}

\def\Y{\mathcal I}
\def\l{L^2(\R^d)}
\def\i{L^\infty(\R^d)}
\def\c{C^\infty_0 (\R^d)}

\def\la{\langle}
\def\ra{\rangle}

\def\qvp{q_{{}_{V_+}}}
\def\qvm{q_{{}_{V_-}}}

\def\1{\mathfrak{1}}
\def\0{\mathfrak{0}}

\def\<{\langle}
\def\>{\rangle}
\def\dbar{\;\;\bar{}\!\!\!d}

\def\Op{\mathfrak{Op}^A}

\newtheorem{lemma}{Lemma}[section]
\newtheorem{corollary}[lemma]{Corollary}
\newtheorem{theorem}[lemma]{Theorem}
\newtheorem{proposition}[lemma]{Proposition}
\newtheorem{definition}[lemma]{Definition}

\newtheoremstyle{remark_rm}{}{}{\upshape}%
{}{\bf}{.}{ }{}

\theoremstyle{remark_rm}
\newtheorem{remark}[lemma]{Remark}
\numberwithin{equation}{section}

\begin{document}

\title{Unicity of the Integrated Density of States for Relativistic Schr\"odinger Operators with
Regular Magnetic Fields and  Singular Electric Potentials}

\date{\today}

\author{Viorel Iftimie\footnote{Institute
of Mathematics Simion Stoilow of the Romanian Academy, P.O.  Box 1-764,
Bucharest, Romania.\hspace{5cm} Email: Viorel.Iftimie@imar.ro} , Marius M\u
antoiu\footnote{Universidad de Chile, Las Palmeras 3425, Casilla 653, Santiago
Chile.} \footnote{Institute
of Mathematics Simion Stoilow of the Romanian Academy, P.O.  Box 1-764,
Bucharest, Romania.\hspace{5cm} Email: Marius.Mantoiu@imar.ro} \ and Radu Purice\footnote{Institute of
Mathematics Simion Stoilow of the Romanian Academy, P.O.  Box 1-764, Bucharest,
Romania.\hspace{5cm} Email: Radu.Purice@imar.ro} \footnote{Laboratoire Europ\'{e}en Associ\'{e}
CNRS {\it Math-Mode}, Franco-Roumain}}

\maketitle \vspace{-1cm}

\begin{abstract}
We show coincidence of the two definitions of the integrated density of states (IDS) for a class of relativistic
Schr\"odinger operators with magnetic fields and scalar potentials introduced in \cite{IMP1,IMP2},
the first one relying on the eigenvalue
counting function of operators induced on open bounded sets with Dirichlet boundary conditions, the other
one involving the spectral projections of the operator defined on the entire space. In this way
one generalizes the results of \cite{DIM,I} for non-relativistic operators.
The proofs needs the magnetic pseudodifferential calculus
developed in \cite{IMP1}, as well as a Feynman-Kac-It\^ o formula for L\'evy processes \cite{IT2,IMP2}.
In addition, in case when both the
magnetic field and the scalar potential are periodic, one also proves the existence of the IDS.
\end{abstract}

\section{Introduction}

We specify first the class of operators we consider. For $d\ge 2$ we set
\begin{equation*}
BC^\infty(\R^d):=\{f\in C^{\infty}(\R^d)\mid\,\partial^{\alpha}f\in L^{\infty}(\R^d),\ \forall\,\alpha\in\N^d\},
\end{equation*}
and
\begin{equation*}
C^\infty_{{\rm pol}}(\R^d):=\{f\in C^{\infty}(\R^d)\mid\,\partial^{\alpha} f {\rm \ is\ polynomially\ bounded},\
\forall\,\alpha\in\N^d\}.
\end{equation*}
The magnetic field $B=\frac{1}{2}\sum^d_{j,k=1}B_{jk}dx_j\wedge dx_k$ satisfies:
\begin{description}
 \item[Hypothesis (i):] $dB=0,\ \ \ B_{jk}=-B_{kj}\in BC^{\infty}(\R^d).$
\end{description}

Using the transversal gauge one constructs a vector potential $A=\sum^d_{j=1}A_jdx_j$, with $A_j\in
C^\infty_{{\rm pol}}(\R^d)$, such
that $dA=B$. The circulation of $A$ through the segment $[x,y],\ x,y\in\R^d$, can be written as
\begin{equation}\label{writ}
\int_{[x,y]}A=-<x-y,\,\Gamma^A(x,y)>,\ \ \ \ \Gamma^A(x,y):=\int^1_0 dsA((1-s)x+sy).
\end{equation}
In some papers \cite{MP1, KO} one proposes the following quantization of a classical observable
$a:T^*\R^d\to\R$:
\begin{equation}\label{of}
\left[\Op (a)u\right](x):=\underset{\R^d}{\int}\underset{\R^d}{\int}dy \dbar\, \xi\,e^{i<x-y,\xi+\Gamma^A(x,y)>}\,a
\left(\frac{x+y}{2},\xi\right)u(y),
\end{equation}
where $u\in\S(\R^d),\,\dbar\,\xi:=(2\pi)^{-d}d\xi$ and the oscillatory integral makes sense if, for example,
$a\in S^m(\R^{d})$.

A symbolic calculus for the operators defined by (\ref{of}), essential for the present work, has been
developed in \cite{IMP1}.
 The quantization (\ref{of}) has the important physical property of being gauge covariant:
 if $\varphi\in C^{\infty}_{\rm pol}(\R^d)$, then
$A$ and $A'=A+d\varphi$ define the same magnetic field $B$ and
$$\mathfrak{Op}^{A'}(a)=e^{i\varphi}\Op (a)e^{-i\varphi}.$$
There exists another approach for quantization in the presence of a magnetic field
\cite{GMS, Ic1, Ic2, IT1, IT2, NU, Pa}.
One defines $\mathfrak{Op}_A(a)$ by the Weyl quantization of the symbol
$$T^*\R^d\ni (x,\xi)\mapsto a\left(x,\xi-A(x)\right)\in\R,$$
but in this way gauge covariance is lost, as shown in \cite{IMP1} for $a(\xi)=<\xi>:=(1+|\xi|^2)^{1/2}$.
One notices, however, that both
quantizations lead to the same magnetic non-relativistic Schr\"odinger operator.

We are concerned in the present paper with the case
\begin{equation}\label{case}
a(\xi)=<\xi>-1
\end{equation}
for which the two quantizations do not coincide.

As shown in \cite{IMP1,IMP2}, the operator $\Op (a)$ in $L^2(\R^d)$ is essentially self-adjoint on $\S(\R^d)$.
One denotes by $H_A$ its closure; then $H_A\ge 0$ and its domain is the magnetic Sobolev space of order $1$:
$$\H^1_A:=\{u\in L^2(\R^d)\mid (D_j-A_j)u\in L^2(\R^d),\ 1\le j\le d\}.$$
We call $H_A$ {\it the relativistic Schr\"odinger operator with magnetic field}. One should remark that another
candidate exists for this concepts, the operator $[(D-A)^2+1]^{1/2}-1$ (cf. \cite{FLS} for instance),
but this one cannot be deduced from a quantization which systematically applies to a whole class of symbols.

For the scalar potential $V$, let us first consider the following condition.
\begin{description}
 \item[Hypothesis (ii):] $V:\R^d\to\R,\ V=V_+ -V_-,\ V_{\pm}\ge 0,\ V_{\pm}\in L^1_{\rm loc}(\R^d)$, and the operator of multiplication by $V_-$ is form-bounded with respect to $H_0$, with relative bound strictly
less than $1$.
 \end{description}

In other words, there exist $\alpha\in [0,1)$ and $\beta\ge 0$ such that
\begin{equation}\label{mult}
\underset{\R^d}{\int}V_{-}|u|^2dx\le\alpha\|H^{1/2}_0 u\|^2+\beta \|u\|^2,\ \ \ \ u\in D(H^{1/2}_0)=\H^{1/2}(\R^d),
\end{equation}
where $\|\cdot\|$ is the norm of $L^2(\R^d)$ and $\H^s(\R^d)$ is the usual Sobolev space of order $s\in\R$.

We are going to show in Section 4 that under the assumptions (i) and (ii), the form sum
$$H\equiv H(A;V):=H_A\overset{\cdot}{+} V$$
is well-defined.
The operator $H$ will be self-adjoint and lower semi-bounded in $L^2(\R^d)$. In particular, $H_A=H(A;0)$.

To use the Feynman-Kac-It\^o formula from Section 4 we will need a stronger hypothesis, involving Kato's class $\K_d$
associated to the operator $H_0$, defined as follows: The semigroup generated by $H_0$ is given by convolution
with a function $p_t$ (defined in Section 3); a function $W\in L^1_{\rm loc}(\R^d),\ W\ge 0$, belongs to
$\K_d$ if
\begin{equation}\label{cato}
\underset{t\searrow 0}{\lim}\,\underset{x\in\R^d}{\sup}\int^t_0\left[\int_{\R^d}p_s (x-y)W(y)dy\right]ds=0.
\end{equation}
In particular, if $W\in L^{\infty}(\R^d),\ W\ge 0$, then $W\in\K_d$. In \cite{vC, CMS, DvC} one shows that
$W\in\K_d$ verifies (\ref{mult}) for any $\alpha >0$.

For our main results we shall need a stronger assumption on $V$.
\begin{description}
 \item[Hypothesis (ii$^\prime$):] $V:\R^d\to\R,\ V=V_+-V_-,\ \ V_{\pm}\ge 0,\ \ V_{\pm}\in L^2_{\rm loc}(\R^d)\ \ {\rm and}\ V_-\in\K_d$.
 \end{description}
To define the integrated density of states (IDS) we need a family $\F$ of bounded open subsets of $\R^d$, satisfying:
\begin{description}
 \item[Hypothesis (iii):] For any $m\in\N^*$, there exists $\Omega\in\F$ such that the ball $B(0;m)$ centered in the origin, of radius $m$, is
contained in $\Omega$.
 \end{description}
\begin{description}
 \item[Hypothesis (iv):] For any $\epsilon >0$, there exists $m_0\in\N^*$ such that if $\Omega\in\F$ and $B(0,m_0)\subset\Omega$, we have
$$
|\{x\in\R^d\mid {\rm dist}(x,\partial\Omega)<1\}|<\epsilon\,|\Omega|,
$$
where we set $|\Omega|$ for the Lebesgue measure of $\Omega$.
 \end{description}

Let us mention some basic references concerning IDS \cite{CL,CFKS,FP,HS,Sh} and \cite{DIM} that is closer related to our work. There are two definitions of IDS. The first one \cite{CL, CFKS} uses the operator $H_{\Omega}$
induced by $H$ on
$\Omega\in\F$, with Dirichlet boundary conditions (it is defined in Section 6, where we prove that
$H_\Omega$ has compact resolvent on $L^2(\Omega)$). IDS is the function
\begin{equation}\label{ids}
\rho :\R\to\R_+,\ \ \ \rho (\lambda):=\underset{\Omega\in\F}{\underset{\Omega\to\R^d}{\lim}}\frac{N_\Omega
(\lambda)}{|\Omega|},
\end{equation}
where $N_\Omega (\lambda)$ is the number of eigenvalues of $H_\Omega$ smaller than $\lambda$.

The second definition \cite{CFKS,HS} uses the fact (proved in Section 5) that the operator $\boldsymbol{1}_{\Omega}
E_\lambda (H) \boldsymbol{1}_\Omega$ belongs to $\mathcal I_1$, i.e. is trace-class.
Here $\boldsymbol{1}_\Omega$ is the operator of multiplication
by the characteristic function of $\Omega$, and $E_\lambda (H)$is the spectral projection of $H$ corresponding to
the interval $(-\infty,\lambda],\ \lambda\in\R$. Then IDS is also defined by
\begin{equation}\label{defg}
\rho (\lambda):=\underset{\Omega\in\F}{\underset{\Omega\to\R^d}{\lim}}\frac{{\rm tr}[\boldsymbol{1}_\Omega E_\lambda
(H)\boldsymbol{1}_\Omega]}{|\Omega|}.
\end{equation}
The existence of the limits (\ref{ids}) and (\ref{defg}) and their equality are both non-trivial problems.

In order to solve them one uses the notion of density of states for $H$, for which we also have
two different definitions. We are going to see in Sections 5 and 6 that for any $f\in C_0(\R)$
(continuous function with compact support on $\R$)
the operators $f(H_\Omega)$ and $\boldsymbol{1}_\Omega f(H)\boldsymbol{1}_\Omega$ belong to $\mathcal I_1$. By the Riesz-Markov
Theorem for any $\Omega\in\mathcal F$ there exist Borel measures $\mu^D_\Omega$ and $\mu_\Omega$  on $\R$, such that
\begin{equation*}
|\Omega|^{-1}{\rm tr}f(H_\Omega)=\int_\R fd\mu^D_\Omega,\ \ \ \ \ |\Omega|^{-1}{\rm tr}\left[\boldsymbol{1}_\Omega
f(H)\boldsymbol{1}_\Omega\right]=\int_\R fd\mu_\Omega.
\end{equation*}
One notices that the two expressions in (\ref{ids}) and (\ref{defg}) are exactly the distribution functions of these
two measures:
\begin{equation*}
\mu^D_\Omega((-\infty,\lambda])=|\Omega|^{-1}N_\Omega(\lambda),\ \ \ \ \ \mu_\Omega((-\infty,\lambda])=
|\Omega|^{-1}{\rm tr}\left[\boldsymbol{1}_\Omega E_\lambda(H)\boldsymbol{1}_\Omega\right].
\end{equation*}
If Borel measures $\mu^D,\mu$ on $\R$ exists such that
\begin{equation*}
\underset{\mathcal F\ni\Omega\rightarrow\R^d}{\lim}\,\mu^D_\Omega=\mu^D,\ \ \ \ \
\underset{\mathcal F\ni\Omega\rightarrow\R^d}{\lim}\,\mu_\Omega=\mu,
\end{equation*}
meaning that for any $f\in C_0(\R)$ and any $\epsilon>0$ there exists $m_0\in\mathbb N^*$ such that if
$B(0;m_0)\subset\Omega$, then
\begin{equation*}
\left|\int_\R fd\mu^D_\Omega-\int_\R fd\mu^D\right|<\epsilon,\ \ \ \ \ \left|\int_\R fd\mu_\Omega-\int_\R
fd\mu\right|<\epsilon,
\end{equation*}
each of them is called {\it the density of states of} $H$. The main result of this article is the
equivalence of these definitions:

\begin{theorem}\label{main}
Under assumptions (i), (ii'), (iii) and (iv), the density of states
$\mu^D$ exists if and only if the density of states $\mu$ exists.
In addition, if one of them exists, then $\mu^D=\mu$.
\end{theorem}

For the non-relativistic Schr\"odinger operator, such a result has been obtained in \cite{DIM} for $V_-=0$ and
in \cite{I} for $V_-\ne 0$. There are several results concerning the existence and unicity of
IDS for non-relativistic Schr\"odinger operators without magnetic field (see \cite{DIM} for
references). The case of a constant magnetic field has been treated in \cite{HS}
($V\in C^\infty(\R^d)$ periodic) or in \cite{HLMW} ($V$ random potential, eventually unbounded from below). We also mention the results in \cite{B1,B2,BR,CHKR,HS} and references therein.

The existence of IDS has been proved in \cite{I} for periodic magnetic fields and scalar potentials.
The second goal of this paper is to extend this last result to the relativistic case.

We  consider a lattice $\Gamma$ in $\mathbb R^d$, generated by a base $e_1,\dots,e_d$:
\begin{equation*}
\Gamma=\{\sum_{j=1}^d\alpha_je_j\mid \alpha_j\in\mathbb Z,\ 1\le j\,\le d\}.
\end{equation*}
Let us also denote by $F$ a fundamental domain of $\R^d$ with respect to $\Gamma$; for instance
\begin{equation*}
F=\{\sum_{j=1}^dt_je_j\mid 0\le t_j<1,1\le j\,\le d\}.
\end{equation*}
We make the following hypothesis:

(v) $V$ and $B_{jk},\,1\le j,k\le d$ are $\Gamma$-periodic functions.

\begin{theorem}\label{riodic}
Under the hypothesis (i), (ii'), (iii), (iv) and (v), the integrated density of states of $H$ exists and for each
$f\in C_0(\R)$ we have
\begin{equation}\label{iodic}
\underset{\mathcal F\ni\Omega\rightarrow\R^d}{\lim}\ |\Omega|^{-1}{\rm tr}\left[\boldsymbol{1}_\Omega f(H)\boldsymbol{1}_\Omega\right]=|F|^{-1}{\rm tr}_\Gamma f(H),
\end{equation}
where ${\rm tr_\Gamma}$ is the $\Gamma$-trace in the sense of Atiyah \cite{A}.
\end{theorem}

The plan of this paper is as follows:

In Section 2 we review first some properties of the magnetic pseudodifferential calculus, established
in \cite{IMP1,IMP2}. Some refined result about commutators are obtained and one overlines approximation
by regularisations (using the magnetic convolution) and cut-offs.

In Section 3 we present the Feller semi-group defined by the free Hamiltonian $H_0$ and the associated L\'evy process.
The diamagnetic inequality (\ref{3.10}) will be a consequence of a Feynman-Kac-It\^o formula for the
relativistic Hamiltonian $H_A$.

Section 4 is devoted to the construction of the relativistic Schr\"odinger operator $H=H(A;V)$.
Using the Feynman-Kac-It\^o formula, representing the semi-group generated by $H$, we prove the important fact that $C^\infty_0(\mathbb{R}^d)$ is an essential domain for the form associated to $H$ and we present some
consequences regarding commutators and covariance under gauge transformations.

In Section 5 we estimate the trace norm of some operators of the form $\boldsymbol{1}_\Omega f(H)\boldsymbol{1}_\Omega$, $\Omega$
being a bounded open subset of $\R^d$ and $f:\R\rightarrow\C$ a suitable function.
The hypothesis $V\in\mathcal K_d$ is essential, allowing us to use some
estimations for the integral kernel of the semi-group generated by $H(0;-V_-)$ (cf. \cite{vC,DvC}).

In Section 6 one defines $H_\Omega$ as a pseudo-self-adjoint operator on $L^2(\R^d)$, using a result
in \cite{S1} on monotone
sequences of quadratic forms. One also estimates the $\mathcal I_1$-norm of operators of the form $f(H_\Omega)$.

Section 7 is dedicated to the proof of Theorem \ref{main}. The main difficulty is the $\mathcal I_1$-norm
estimate of operators of the form $\boldsymbol{1}_\Omega(H+\lambda)^{-m}\boldsymbol{1}_\Omega-(H_\Omega+\lambda)^{-m}$.
Then, using ideas from \cite{DIM} finishes the proof.

In Section 8 we prove Theorem \ref{riodic}, on the lines of the proof of Theorem 1.6 from \cite{I}.

\section{The magnetic pseudodifferential calculus}\label{secintro}

Let us recall first some properties of operators defined by (\ref{of}), proved in \cite{IMP1}. We are  going to
assume everywhere that $B=dA$ fulfills hypothesis (i).

\begin{definition}\label{Sm}
Let $m\in\R$.

(a) A function $f\in C^\infty(\R^{2d})$ belongs to the symbol space $S^m(\R^d)$ if for any $\alpha,\beta\in\mathbb N^d$
there is a constant $C_{\alpha,\beta}>0$ such that
\begin{equation*}
|\left(\partial^\alpha_x\partial^\beta_\xi f\right)(X)|\le C_{\alpha,\beta}<\xi>^{m-\beta},\ \ \ \ \ X=(x,\xi)\in\R^{2d}.
\end{equation*}
The space $S^m(\R^d)$ is endowed with the usual Fr\'echet topology.

(b) $S^{-\infty}(\mathbb R^d):=\cap_{m\in\R}S^m(\R^d)$ is endowed with the projective limit topology.

(c) A symbol $f\in S^m(\R^d)$ is called {\rm elliptic} if for some positive constants $C,R$ one has
\begin{equation*}
|f(X)|\ge C<\xi>^m,\ \ \ \ \ \forall X=(x,\xi)\in\R^{2d},\ \ |\xi|\ge R.
\end{equation*}

(d) We call {\rm principal symbol} of an operator of the form $\Op(f)$, where $f\in S^m(\R^d)$, any element
$f_0\in S^m(\R^d)$ such that $f-f_0\in S^{m-1}(\R^d)$.
\end{definition}

\begin{proposition}\label{rima}
Let $f\in S^m(\R^d)$ and $g\in S^{m'}(\R^d)$.

(a) $\Op(f)$ is a continuous linear operator on $\mathcal S(\R^d)$ and on $\mathcal S'(\R^d)$.

(b) $\Op(\overline f)$ is the formal adjoint of $\Op(f)$, i.e.
$$
\left(\Op(f)u,v\right)_{L^2(\R^d)}=\left(u,\Op(\overline f)v\right)_{L^2(\R^d)},\ \ \ \ \ \forall u,v\in \mathcal S(\R^d).
$$

(c) There exists a unique element $f\circ^B g\in S^{m+m'}(\R^d)$ such that
$$\Op(f)\Op(g)=\Op(f\circ^B g).$$
Moreover, a principal symbol of $\Op(f)\Op(g)$ is $fg$.
\end{proposition}

\begin{proposition}\label{aoa}
Let $f\in S^0(\R^d)$. Then $\Op (f)\in B(L^2(\R^d))$, and its norm in $B(L^2(\R^d))$ is dominated by a semi-norm of $f$ in $S^0(\R^d)$.
\end{proposition}

\begin{definition}
Let $s\in\R_+,\ p_s (\xi):=<\xi>^s\in S^s(\R^d),\ P_s:=\Op(p_s).$

a) An element $u\in L^2(\R^d)$ belongs to the magnetic Sobolev space $\H^s_A (\R^d)$ if $P_s u\in L^2(\R^d)$.
$\H^s_A(\R^d)$ is a Hilbert space for the norm
\begin{equation}\label{norm}
\|u\|_{s,A}:=\left( \|P_s u\|^2_{L^2(\R^d)}+\|u\|^2_{L^2(\R^d)}\right)^{1/2},
\end{equation}
and $\mathcal S(\R^d)$ is dense in $\H^s_A(\R^d)$.

b) $\H^{-s}_A(\R^d)$ will be the dual of $\H^s_A(\R^d)$ endowed with the natural norm.
\end{definition}

\begin{remark} If $s\in\N$, we have
$$\H^s_A(\R^d)=\{u\in L^2(\R^d)\mid\,(D-A)^{\alpha} u\in L^2(\R^d),\ \forall\,\alpha\in\N^d,\,|\alpha|\leq s\},$$
and a norm equivalent to (\ref{norm}) is given by
$$\|u\|'_{s,A}=\left(\underset{|\alpha|\leq s}{\sum}\|(D-A)^{\alpha}u\|^2_{L^2(\R^d)}\right)^{1/2}.$$
We used the notation $(D-A)^{\alpha}=(D_1-A_1)^{\alpha_1}\dots (D_n-A_n)^{\alpha_n}$.
\end{remark}

\begin{proposition}\label{for}
For each $s,m\in\R$ and each $f\in S^m(\R^d)$,
$$\Op (f)\in B\left(\H^s_A(\R^d),\H^{s-m}_A(\R^d)\right).$$
\end{proposition}

\begin{proposition}\label{let}
Let $p\in S^m(\R^d)$ be real and elliptic, $m\geq 0$. We assume $p(X)\geq 0$ for $|\xi|\geq R$ ($R>0$ large enough).
Then the operator $\Op (p)$, defined on $S(\R^d)$, is essentially self-adjoint in $L^2(\R^d)$. Its closure
$P$ will be a lower semi-bounded self-adjoint operator on the domain $\H^m_A(\R^d)$.
\end{proposition}

\begin{remark} This proposition applies to the case $p(X)=<\xi>-1$.
The corresponding operator, denoted by $H_A$, will be a lower semi-bounded self-adjoint operator on
$L^2(\R^d)$ with domain
\begin{equation}\label{has}
\H^1_A(\R^d)=\{u\in L^2(\R^d)\mid\, (D_j-A_j)u\in L^2(\R^d),\,1\leq j\leq d\}.
\end{equation}
 In fact, adapting arguments from \cite{Ic2} (where the quantification $\mathfrak{Op}_A$ is used),
 one can show that $H_A\geq 0$.
 \end{remark}

 The next result has been proved in \cite{IMP1} for $B$ admitting a vector potential with bounded derivatives
and for the general case of hypothesis (i) in \cite{IMP3}.

\begin{proposition}\label{verif}
Let us consider verified the hypothesis of Proposition \ref{let}.

a) If $\lambda\in\R,\ \lambda <\inf\sigma (P)$, then $(P-\lambda)^{-1}$ is the closure in $\l$ of an operator
$\Op (p_{(\lambda)})$, with $p_{(\lambda)}\in S^{-m}(\RR) $. If in addition $\lambda\le {\rm inf}p-1$,
then a principal symbol of $\Op (p_{(\lambda)})$ is $(p-\lambda)^{-1}$.

b) If $m>0,\ p\ge 1$ and $P\ge 1$, then for every $s\in\R,\ P^s$ is the closure in $\l$ of a  operator $\Op (q_{(s
)}),\ q_{(s)}\in S^{sm}(\RR)$, which admits $p^s$ like principal symbol.
\end{proposition}

In the remaining part of this section we are going to prove two properties of commutators of magnetic pseudo-differential
operators, as well as applications to approximation by regularization or cut-off.

\begin{proposition}\label{indep}
Let $m\in\mathbb{R}$ or $m=-\infty$ and $g\in S^{m'}(\RR),\ m'\in\mathbb{R}$. Then

a) $f\circ^B g-g\circ^B f\in S^{m+m'-1}(\RR),\ \ \forall\,f\in S^m(\RR)$.

b) Let $\overline m\in\R$ and $M$ be a subset of $S^m(\RR)$ formed by symbols independents of $x\in\RR$, such that
the  set $\{\partial_{\xi_1}f,\dots,\partial_{\xi_d}f\mid\,f\in M\}$ is bounded in $S^{\overline m}(\RR)$. Then the
set $\{f\circ^B g- g\circ^B f\mid f\in M\}$ is bounded in $S^{\overline m+m'}(\RR)$.
\end{proposition}

\begin{proof}
a) It follows from the Proposition 2.2 (c).

To verify b), one uses the composition formula from \cite{IMP1} for
$f\in M$, in which the integral is oscillatory:
\begin{equation}\label{2.3}
(f\circ^B g)(X)=\int_{\R^{2d}} \int_{\R^{2d}} \dbar Y\dbar Z\,e^{-2i[Y,Z]}\,\omega^B (x,y,z)\,f(\xi-\eta)\,g(X-Z),
\ \ \ X\in\R^{2d},
\end{equation}
where $X=(x,\xi),\ Y=(y,\eta),\ Z=(z,\zeta)$ are points in $\R^{2d}$, $\dbar Y=\pi^{-d}dY$,
$[Y,Z]=<\eta,z>-<\zeta,y>$ ($<\cdot,\cdot>$ is the scalar product in $\RR$) and $\omega^B (x,y,z)=e^{-4iF_B(x,y,z)}$,
where $F_B\in C^\infty (\R^{3d})$, depending only on the magnetic field $B$ and its first order derivatives, are of
the form:
\begin{equation}\label{2.4}
\underset{1\le j\le d}{\sum}\left[D_j(x,y,z)y_j+E_j(x,y,z)z_j\right],\ \ \ D_j,\,E_j\in BC^\infty (\R^{3d}).
\end{equation}
Using the Leibnitz-Newton formula
$$
  f(\xi-\eta)=f(\xi)-\underset{1\le j\le d}{\sum}\eta_j\int^1_0(\partial_j f)(\xi-t\eta)dt
$$
and the fact that $1\circ^B g=g$, one writes (\ref{2.3}) as
\begin{equation}\label{2.5}
(f\circ^B g)(X)=f(\xi)g(X)+\rho_f (X),
\end{equation}
where
\begin{equation}\label{2.6}
\rho_f (X)=-\underset{1\le j\le d}{\sum}\int^1_0dt\int_{\R^{2d}} \int_{\R^{2d}}\dbar Y\dbar Z\,\eta_j\,e^{-2i[Y,Z]}
\,\omega^B (x,y,z)\,(\partial_j f)\,(\xi-t\eta)g(X-Z).
\end{equation}
We use the identity
$$\eta_j e^{-2i[Y,Z]}=-\frac{1}{2i}\partial_{z_j}\left(e^{-2i[Y,Z]}\right)$$
to integrate by parts with respect to $z_j$.
We also use (\ref{2.4}) as well as
$$y_k e^{-2i[Y,Z]}=\frac{1}{2i}\partial_{\zeta_k}\left(e^{-2i[Y,Z]}\right),\ \ \
z_k e^{-2i[Y,Z]}=-\frac{1}{2i}\partial_{\eta_k}\left(e^{-2i[Y,Z]}\right)$$
to integrate by parts with respect to $\zeta_k$ and $\eta_k$. This gives
$$
\rho_f (X)=\underset{1\le j\le d}{\sum}\int^1_0 dt\int_{\R^{2d}}\int_{\R^{2d}} \dbar Y \dbar Z\,e^{-2i[Y,Z]}\,
\omega^B(x,y,z)\big[ \underset{1\le k\le d}{\sum}D_{jk}(x,y,z)(\partial_{\xi_k}g)(X-Z)
$$
$$
(\partial_j f)(\xi-\eta)+t\underset{1\le k\le d}{\sum}E_{jk}(x,y,z)(\partial_j\partial_k f)(\xi-t\eta)g(X-Z)-
$$
$$
-(\partial_j f)(\xi-t\eta)(\partial_{x_j}g)(X-Z) \big],\ \ \ \ D_{jk},\,E_{jk}\in BC^\infty (\R^{3d}).
$$
By hypothesis, the sets $\{\partial_j f\mid f\in M\},\ 1\le j\le d$ are bounded in $S^{\overline m}(\RR)$.

Using the standard integration by parts procedure  with respect to $y,z,\eta,\zeta$, starting from the equality
$$<\eta>^{2N}e^{-2i[Y,Z]}=\left(1-\frac{1}{4}\Delta_z\right)^Ne^{-2i[Y,Z]},\ \ N\in\N$$
and its analogs, by eliminating
the monomials in $y$ and $z$, as above, one obtains the estimation
$$
  \left|\rho_f (X)\right|\le p(f)\,q(g)\int^1_0 dt\int_{\R^{2d}}\int_{\R^{2d}}\dbar Y\dbar Z <z>^{-2N_\eta}<\zeta>^{-2N_y}
  <\eta>^{-2N_z}<y>^{-2N_\zeta}\cdot
$$
$$
  \cdot<\xi-t\eta>^{\overline m}<\xi-\zeta>^{m'},
$$
where $N_\eta,\,N_y,\,N_z,\,N_\zeta$ are natural integers which must be chosen in a suitable way in order to have absolute convergence of the integrals, $p(f)=\underset{1\le j\le d}
{\sum}p_j(\partial_j f)$, $p_j$ is a continuous semi-norm on $S^{\overline m}(\RR)$ and $q$ a continuous semi-norm on
$S^{m'}(\RR)$.
Since
$$
  <\xi-t\eta>^{\overline m}\le C<\xi>^{\overline m}<\eta>^{|\overline m|},\ \ \ <\xi-\zeta>^{m'}\le C<\xi>^{m'}
  <\zeta>^{|m'|},\quad C\in\mathbb{R}_+,
$$
one can choose $N_\eta=N_\zeta=d,\,\ N_y=d+|m'|,\,\ N_z=d+|\overline m|$ and get
  $$
    \left|\rho_f (X)\right|\le C_0\,p(f)\,q(g)<\xi>^{\overline m+m'} ,\ \ \ \ C_0>0\ {\rm constant}.
  $$
Analogously one estimates the derivatives of $\rho$ and obtains that the set $\{\rho_f\mid\,f\in M\}$ is bounded in
$S^{\overline m+m'}(\RR)$.

In the same way one can show that $g\circ^B f=fg+\rho'_f$ and $\{\rho'_f;\,f\in M\}$
is bounded in $S^{\overline m+m'}(\RR)$.
\end{proof}

\begin{definition}\label{magn}
Let $u\in\S'(\RR)$, $f\in \S(\RR)$. One calls {\rm magnetic convolution of $u$ with} $f$, the function $u\star^A f
\in C^\infty (\RR)$ defined by
\begin{equation}\label{2.7}
(u\star^A f)(x):=<u(y), \,e^{i<x-y,\Gamma^A (x,y)>}f(x-y)>,\ \ \ x\in\RR,
\end{equation}
where $<\cdot,\cdot>$ is the duality bracket between $\S(\RR)$ and $\S'(\RR)$ .
\end{definition}

\begin{remark} Using the equality $f(x-y)=\int_{\RR}e^{i<x-y,\xi>}\widehat f(\xi)d\xi$, one finds out  that
$u\star^A f=\Op (\widehat f)u$.
\end{remark}

To regularize a distribution by means of the magnetic convolution, one uses a standard $\delta$-sequence.
Let us consider a function $\theta\in\c,\ \theta\ge 0$, ${\rm supp}\,\theta\subset B(0;1),\ \ \int_{\RR}
\theta (x)dx=1$. For $j\ge 1$ one defines $\theta_j(x):=j^d\theta(jx),\ x\in\RR$.
Then $\theta_j\in\c,\ {\rm supp}\,\theta_j
\subset B(0;1/j),\ \widehat\theta_j\in\S(\RR),\ \widehat\theta_j (\xi)=\widehat \theta(j^{-1}\xi),\ \xi\in\RR$.
For $u\in\S'(\RR)$ we set $R_ju:=u\star^A\theta_j$.

\begin{proposition}\label{if}
(a) If $u\in\mathcal E'(\RR),\ R_ju\in\c$ and
${\rm supp}R_ju\subset \{x\in\RR\mid\,{\rm dist}(x,{\rm supp}\,u)\le 1/j\}.$

(b) If $u\in L^\infty (\RR),\ R_ju\in L^\infty (\RR)$ and $\|R_ju\|_{L^\infty (\RR)}\le \|u\|_{L^\infty (\RR)}$.

(c) If $u\in\l,\ R_ju\in\l$ and $\underset{j\to\infty}{\lim}R_ju=u$ on $\l$.

(d) Let $P=\Op (p),\ p\in S^{1/2}(\RR)$. If $u\in\l$ and $Pu\in\l$, then
$\underset{j\to\infty}{\lim}PR_ju=Pu$ in $\l$.
\end{proposition}

\begin{proof}
Properties (a) and (b) are evident.

(c) Since
$\widehat \theta\in\S(\RR)\subset S^{-\infty}(\RR)\subset S^0 (\RR)$
and $R_j=\Op (\widehat\theta_j)$, Proposition 2.3 shows that $R_ju\in\l$. But
$$
  \left(R_j\right)u(x)=\int_{\RR}\theta (z)\,u\left(x-z/j\right)\,e^{-(i/j)<z,\Gamma^A(x,x-z/j)>}dz
$$
and
$
  u(x)=\int_{\RR}\theta (z)\,u(x)\,dz.
$
Using twice the Dominated Convergence Theorem and the continuity of $u$ in mean, one gets
$$
\|R_ju-u\|_{\l}\le\int_{\RR}\theta (z)\big[ \|u(\cdot-z/j)-u(\cdot)\|_{\l}+
$$
$$
 + \|u(\cdot)\left(e^{-(i/j)<z,
\Gamma^A(\cdot,\cdot-z/j)>}-1\right)\|_{\l}\big]dz\underset{j\to\infty}{\to}0.
$$

(d) The sequence $\{j^{1/2}\partial_k\widehat\theta_j\}_{j\ge 1}$ is bounded  in $S^{-1/2}(\RR)$ for $1\le k\le d$.
One applies Proposition 2.10 (b) with
$M=\{j^{1/2}\widehat\theta_j;\,j\ge 1\}$,$\ m=-\infty,\ \overline m=-\frac{1}{2}$,
$m'=\frac{1}{2}$, $f(\xi)=j^{1/2}\widehat\theta_j (\xi)$, $g=p$
and deduces that the set
$$\{j^{1/2}(\widehat\theta_j
\circ^B p-p\circ^B\widehat\theta_j)\mid\,j\ge 1\}$$
is bounded in $S^0(\RR)$. By Proposition 2.3, there is a constant
$C>0$ such that
$$\|R_jP-PR_j\|_{B(\l)}\le Cj^{-1/2},\ j\ge 1.$$
Thus
$\underset{j\to\infty}{\lim}\big(PR_ju-R_jPu\big)=0$ in $\l$ and
$\underset{j\to\infty}{\lim}R_jPu=Pu$ in $\l$ implies the conclusion.
\end{proof}

\begin{proposition}\label{const}
Let $P=\Op (p),\ p\in S^m(\RR),\ m\le 1$ and $\varphi\in BC^\infty (\RR)$. Let us suppose that
there exists $N\in\N,\ N>m+d-1$ such that $|M_N|<\infty$, where
$$M_N:=\underset{|\alpha|=N+1}{\cup}{\rm supp}\,\partial^{\alpha}\varphi.$$
Then there is a constant $C>0$ independent on $\varphi$, operators
$S_\alpha,S'_\alpha\in B(\l),\,1\le |\alpha|\le N$, independent
on $\varphi$ and operators $T_N,T'_N\in\Y_2$ (the Hilbert-Schmidt space on $\l$) such that
 \begin{equation}\label{2.8}
 \|T_N\|_{\Y_2}+\|T'_N\|_{\Y_2}\le C \underset{|\alpha|=N+1}{\rm max}\|\partial^\alpha\varphi\|_{L^\infty(\RR)}|M_N|^{1/2}
 \end{equation}
 and
 \begin{equation}\label{2.9}
 [\varphi,P]:=\varphi P-P\varphi=\underset{1\le |\alpha|\le N}{\sum}(\partial^\alpha\varphi)S_\alpha+T_N=
 \underset{1\le |\alpha|\le N}{\sum}S'_\alpha(\partial^\alpha\varphi)+T'_N,
 \end{equation}
 with the convention that the sums in (\ref{2.9}) do not exist if $N=0$.
  \end{proposition}

 \begin{proof}
 Using (1.2) and Taylor's formula
 $$
   \varphi (x)-\varphi (y)=-\underset{1\le |\alpha|\le N}{\sum}\frac{(y-x)^\alpha}{\alpha!}(\partial^\alpha\varphi)(x)+
   r_N(x,y),
 $$
 where
 $$
   r_N(x,y)=-\underset{|\alpha|=N+1}{\sum}\frac{(y-x)^\alpha}{N!}\int^1_0 (1-t)^N(\partial^\alpha\varphi)(x+t(y-x))dt,
 $$
 one obtains the first equality from (\ref{2.9}) with $S_\alpha=-\frac{1}{\alpha!}\Op (D^\alpha_\xi p)$ and
\begin{equation}\label{2.10}
   T_Nu(x):=-\frac{1}{N!}\underset{|\alpha|=N+1}{\sum}\int^1_0 dt(1-t)^N \int_{\RR}dy\dbar\xi\,e^{i<x-y,\xi+\Gamma^A
   (x,y)>}\times
 \end{equation}
 $$
   \times(D^\alpha_\xi p)\left(\frac{x+y}{2},\xi\right)(\partial^\alpha\varphi)(x+t(y-x))u(y),\ \ \ u\in\S(\RR).
 $$
 Since for $|\alpha|\ge 1$ one has $D^\alpha_\xi p\in S^0(\RR)$, by Proposition 2.3 one has $S_\alpha\in B(\l)$.
 But, using in (\ref{2.10}) the identity
 $<x-y>^{2s}e^{i<x-y,\xi>}=(1-\Delta_\xi)^s(e^{i<x-y,\xi>})$, for $s\in\N$,
 to integrate by parts, one sees that for any $s\in\N$, $T_N$ can be written as an integral operator with kernel
 $$
   K_N(x,y):=-\frac{1}{N!}\underset{|\alpha|=N+1}{\sum}\int^1_0 dt(1-t)^N \int_{\RR}\dbar \xi\,e^{i<x-y,\xi+\Gamma^
   A(x,y)>}\times
 $$
 $$
   \times <x-y>^{-2s}\left[(1-\Delta_\xi)^sD^\alpha_\xi p\right]\left(\frac{x+y}{2},\xi\right)(\partial^\alpha\varphi)(x+
   t(y-x)).
 $$
 Denoting by $\boldsymbol{1}_M$ the characteristic function of a set $M\subset\RR$, one obtains for each $s\in\N$
 $$
   |K_N(x,y)|\le C_s<x-y>^{-2s}\underset{|\alpha|=N+1}{\rm max}\|\partial^\alpha\varphi\|_{L^\infty(\RR)}
   \int^1_0 \boldsymbol{1}_{M_N}(x+t(y-x))dt,
 $$
 with $C_s$ a constant  depending  on $s$ but not on $\varphi$. Choosing $s>d/4$ one gets
 $$
   \|T_N\|_{\Y_2}=\|K_N\|_{L^2(\RR\times\RR)}\le
$$
$$
\leq C_s\underset{|\alpha|=N+1}{\rm max}\|\partial^\alpha\varphi\|
   _{L^\infty(\RR)}\ \underset{t\in [0,1]}{\rm sup}\left[\int_{\mathbb{R}^{2d}}\hspace{-0.3cm}<x-y>^{-4s}\boldsymbol{1}_{M_N}(x+t(y-x))\,dx\,dy\right]^{1/2}.
 $$
 The integral in the previous formula equals
 $$
   \int_{\mathbb{R}^{2d}}<z>^{-4s}\boldsymbol{1}_{M_N}(x+tz)dxdz=\int_{\RR}\boldsymbol{1}_{M_N}(y)dy\int_{\RR}<z>^{-4s}dz=
   |M_N|\int_{\RR}<z>^{-4s}dz.
 $$
It follows that the norm $\|T_N\|_{\Y_2}$ is bounded with the right member of the inequality (\ref{2.8}).

In the same
way one gets the second equality from (\ref{2.9}) and the corresponding bound for $\|T'_N\|_{\Y_2}$.
 \end{proof}

A first application of the Proposition 2.14 concerns cut-off approximations.
Let
$\psi\in\c$, $0\le\psi\le 1$, $\text{\rm supp}\,\psi\subset B(0;2)$, $\psi=1$ on $B(0;1)$.
For $j\ge 1$ one sets $\psi_j(x):=\psi(x/j),\ x\in\RR$.

\begin{proposition}\label{lim}
Let $u\in\l$ and  $P=\Op (p),\ p\in S^m(\RR),\ m\le 1$.

a) $\psi_ju\in\l$ and $\underset{j\to\infty}{\lim}\psi_ju=u$ in $\l$.

b) If $Pu\in\l$, then $P(\psi_ju)\in\l$ and $\underset{j\to\infty}{\lim}P(\psi_ju)=Pu$ in $\l$.
\end{proposition}

\begin{proof}
a) is trivial.

b) follows if for some constant $C>0$ one obtains the inequality
\begin{equation}\label{2.11}
             \|[\psi_j, P]\|_{B(\l)}\le Cj^{-1},\ \ \ \ j\le 1.
\end{equation}
For this one applies Proposition 2.14 with $\varphi=\psi_j$ and $N=d+1$.
Since
${\rm supp}\,\partial^\alpha\psi_j\subset B(0;2j)$, $\forall\,\alpha\in\N^d$,
there is a constant $C_1>0$ such that
$|M^{(j)}_N|^{1/2}\le C_1\,j^{d/2}$, $\forall\,j\ge 1,$
where
$M^{(j)}_N:=\underset{|\alpha|=N+1}{\cup}{\rm supp}\,\partial^\alpha\psi_j.$
On the other hand $(\partial^\alpha\psi_j)(x)=j^{-|\alpha|}(\partial^\alpha\psi)(x/j)$, hence there
is  a constant $C_2>0$ such that
$\|\partial^\alpha\psi_j\|_{L^\infty (\RR)}\le C_2\,j^{-|\alpha|}$,
for any $j\ge 1$ and for all $\alpha\in\N^d$ with $1\le |\alpha|\le N+1$.
Then (\ref{2.11}) follows from (\ref{2.8}) and (\ref{2.9}).
\end{proof}

\section{The Feller semigroup}\label{mirabilis}

In this section we are going to recall some well-known properties of the semi-group $\left\{e^{-tH_0}\right\}_{t\ge 0}$,
where $H_0$ is the free relativistic Hamiltonian for $A=0$. $H_0$ is self-adjoint
on $\l,\ H_0\ge 0$ and its domain is the standard Sobolev space $\H^1(\RR)\equiv\H^1_0 (\R^d)$. Its restriction to
$\S(\RR)$ is the operator
$\mathfrak{Op}^0 (p)$, with $p(X)=<\xi>-1$.
By the L\'evy-Khincin formula (see for example \cite{RS}), there exists a measure $n$ on $\RR$ such that
\begin{equation}\label{3.1}
<\xi>-1=-\int_{\RR}\left[e^{i<y,\xi>}-1-i<y,\xi>\boldsymbol{1}_{B(0;1)}(y)\right]dn(y),\ \ \ \forall\,\xi\in\RR.
\end{equation}
Cf. \cite{Ic2} one has the explicit formula
\begin{equation}\label{3.2}
dn(y)=2(2\pi)^{-\frac{d+1}{2}}|y|^{-\frac{d+1}{2}}K_{\frac{d+1}{2}}(|y|)dy,
\end{equation}
where $K_\nu$ is the modified Bessel function of the third kind and order $\nu$, for which the next
inequalities are verified for some positive constant $C$:
\begin{equation}\label{3.3}
0< K_\nu (r)\le C\max \left(r^{-\nu},r^{-1/2}\right)\,e^{-r},\ \ \ \ \forall\,r>0,\ \forall\,\nu>0.
\end{equation}

By \cite{IT2,CMS}, for $t>0$, the operator $e^{-tH_0}$ is given by the convolution with the function
\begin{equation}\label{3.4}
p_t(x):=(2\pi)^{-d}\frac{t}{\sqrt{|x|^2+t^2}}\int_{\RR} e^{t-<\xi>\sqrt{|x^2|+t^2}}d\xi=
\end{equation}
$$
  =2^{-\frac{d-1}{2}}\pi^{-\frac{d+1}{2}}t\,e^t\,(|x|^2+t^2)^{-\frac{d+1}{4}}K_{\frac{d+1}{2}}\left(
  \sqrt{|x|^2+t^2}\right),\ \ \ x\in\RR.
$$
One verifies the properties
\begin{equation}\label{3.5}
p_t(x)>0,\ \ \ \int_{\RR} p_t(x)dx=1
\end{equation}
and the fact that $e^{-tH_0}$ can be extended as an well-defined bounded operator on the Banach
space
$$C_\infty (\RR):=\left\{f\in C(\RR)\mid\underset{|x|\to\infty}{\lim}f(x)=0 \right\},$$
equipped with the norm
$\|\cdot\|_\infty$. One also checks easily the Feller semi-group properties for the family of these extensions.

By \cite{CMS,DvC}, this Feller semi-group is generated by a L\'evy process. More precisely, on the space
$\Omega$ of the "c\`adlag" functions on $[0,\infty)$ ($\RR$-valued, continuous to the right, having left limits),
endowed with the smallest $\sigma$-algebra $\F$ for which all the coordinate functions
$$
\Omega\ni\omega\mapsto X_t(\omega):=\omega (t)\in\RR
$$
are measurable, one can define for each $x\in\RR$ a probabilistic measure $P_x$ such that $P_x\{X_0=x\}=1$ and
the random variables
$X_{t_1}-X_{t_0},\dots,X_{t_n}-X_{t_{n-1}}$
are independent with distributions
$p_{t_1-t_0},\dots,p_{t_n-t_{n-1}}$
for each $0=t_0<t_1<\dots<t_n<\infty$. If we denote by $E_x$ the expectation with respect to the probability
$P_x$, then for any $f\in C_\infty (\RR)$ and $t\ge 0$  we have
\begin{equation}\label{3.6}
\left[e^{-tH_0}f\right](x)=E_x \left(f\circ X_t\right),\ \ \ x\in\RR.
\end{equation}
By the L\'evy-It\^o Theorem \cite{IW,IT2} one has
\begin{equation}\label{3.7}
X_t=x+\int^{t_+}_0 \int_{\RR}y\,\widetilde N_X\,(ds\,dy),
\end{equation}
 where
 $$\widetilde N_X\,(ds\,dy):=N_X(ds\,dy)-\widehat N_X(ds\,dy),$$
 $$\widehat N_X(ds\,dy):=E_x\left(N_X(ds\,dy)\right)=dsdn(y)$$
 and $N_X$ is a "counting measure" on $[0,\infty)\times\RR$, defined by
 $$
   N_X((t,t']\times B):=\sharp\{s\in(t,t']\mid\,X_s\neq X_{s-},\,X_s-X_{s-}\in B\},
 $$
 where $0<t<t'$ and $B$ is a Borel subset of $\RR$.
   Using the procedure from \cite{IT2} (where one works with the quantification $\mathfrak{Op}_A$), one
gets a Feynman-Kac-It\^o formula for the Hamiltonian $H_A$. For $u\in\l,\ x\in
 \RR$ and $t\ge 0$ one has
\begin{equation}\label{3.8}
\left(e^{-tH_A}u\right)(x)=E_x\left((u\circ X_t)\,e^{-S(t,X)}\right),
\end{equation}
where
\begin{equation}\label{3.9}
  S(t,X):=i\int^{t_+}_0\int_{\RR}\widetilde N_X(dsdy)\left\la \int^1_0 A(X_{s-}+ry)dy,y\right\ra +
\end{equation}
$$
+i\int^t_0\int_{\RR} \widehat N_X(dsdy)\left\la \int^1_0 [A(X_{s}+ry)-A(X_s)]dr,y\right\ra .
$$
Let us remark that from (\ref{3.8}) and (\ref{3.6}) one obtains the diamagnetic inequality for the relativistic
hamiltonian $H_A$:
\begin{equation}\label{3.10}
\left|e^{-tH_A}u\right| \le e^{-tH_0} |u|,\ \ \ \forall\,u\in\l,\ \forall\,t\ge 0.
\end{equation}

(\ref{3.10}) implies another proof of the fact that $H_A\ge 0$: $e^{-tH_0}$ is a contraction,
thus $e^{-tH_A}$ is a contraction too, which implies $H_A\ge 0$.

Once again from (\ref{3.10}), it follows that for any $\lambda >0,\,r>0$ and $u\in\l$ one has
\begin{equation}\label{3.11}
\left|(H_A+\lambda)^{-r}u\right|\le (H_0+\lambda)^{-r}|u|.
\end{equation}
 This inequality is deduced using the fact that for any lower semi-bounded self-adjoint operator $H$ in a
 complex Hilbert space $\H$, for any $r>0$ and any $\lambda\in\R$ such that $\lambda+\inf\sigma(H)>0$, one has
 \begin{equation}\label{3.12}
 (H+\lambda)^{-r}=\frac{1}{\Gamma (r)}\int^\infty_0 t^{r-1}e^{-\lambda t}e^{-tH}dt,
 \end{equation}
  where $\Gamma$ is the Euler function of the second kind.

\section{The Hamiltonian $H(A;V)$}\label{position}

We denote by $h_A$ the quadratic form associated to $H_A$:
\begin{equation}\label{quadric}
h_A (u,v):=\left(H^{1/2}_A u,H^{1/2}_A v\right)_{L^2 (\R^d)},\ \ \ u,v\in D(h_A):=D(H^{1/2}_A).
\end{equation}
To a function $W\in L^1_{\rm loc}(\R^d),\,\ W\ge 0$, one assigns a quadratic form $q_W$:
\begin{equation}\label{form}
q_W(u,v):=\int_{\R^d}W(x)u(x)\overline{v(x)}dx,\ \ \ u,v\in D(q_W):=\{f\in L^2(\R^d)\mid W^{1/2}f\in L^2(\R^d)\}.
\end{equation}
These forms are symmetric, closed and positive. We set
$$
h_A(u):=h_A (u,u),\ \ \ q_W(u):=q_W(u,u).
$$
The next result is known \cite{IMP2}, but for convenience we are going to include a proof.

\begin{proposition}\label{sesq}
We assume (i) and (ii). Then the sesquilinear form
$
h=h(A;V):=h_A+\qvp-\qvm
$
is well-defined on $D(h_A)\cap D(\qvp)$, being symmetric, closed and lower semi-bounded.
Thus it defines a lower semi-bounded self-adjoint operator on $ L^2(\R^d)$, denoted by
$H=H(A;V):=H_A\overset{\cdot}{+}V$ (in the sense of forms).
\end{proposition}

\begin{proof}
The form $h_A+\qvp$, defined on $D(h_A)\cap D(\qvp)$, is densely defined, symmetric, closed and positive. The
conclusion of the Proposition would follow if we show that the form $\qvm$ is $(h_A+\qvp)$-bounded, with
relative bound $< 1$.

We denote by $H_+ :=H_A\overset{\cdot}{+}V_+$ the unique self-adjoint operator $\ge 0$ associated to the form $h_A+\qvp$.
Since $C^{\infty}_0 (\R^d)\subset D(h_A)\cap D(\qvp)$ we can use the version from \cite{KM} of Kato-Trotter formula
\begin{equation}\label{kato}
e^{-tH_{_+}}=s-\underset{n\to\infty}{\lim}\left[e^{-\frac{t}{n}H_A}e^{-\frac{t}{n}V_+}\right]^n,\ \ \ \forall\,t\ge 0.
\end{equation}
Combining with (3.10) and (3.12) we infer that for every $r>0,\ \lambda >0\ {\rm and}\ f\in L^2(\R^d)$ one has
\begin{equation}\label{infer}
\left|(H_++\lambda)^{-r}f\right|\le (H_0+\lambda)^{-r}|f|.
\end{equation}
Let $g\in L^2(\R^d),\ \lambda >0$ large enough, $u:=(H_0+\lambda)^{-1/2}g$. By using the assumption (ii),
there exists $\alpha\in (0,1),\ \beta\ge 0$ and $\alpha '\in (0,1)$ such that
\begin{equation}\label{tel}
\qvm (u)\le\alpha \|H^{1/2}_0 u\|^2_{L^2(\R^d)}+\beta\|u\|^2_{L^2(\R^d)}=\alpha\|H^{1/2}_0
(H_0+\lambda)^{-1/2}g\|^2_{L^2(\R^d)}+
\end{equation}
$$
+\beta\|(H_0+\lambda)^{-1/2}g\|^2_{L^2(\R^d)}\le\left(\alpha+\frac{\beta}{\lambda}\right)\|g\|^2_{L^2(\R^d)}
\le\alpha'\|g\|^2_{L^2(\R^d)}.
$$
For $v\in D(h_A)\cap D(\qvp)$ we set $f:=(H_++\lambda)^{1/2}v$ and $g:=|f|$.
Using (\ref{infer}) with $r=1/2$, (\ref{tel}) and the explicit form of $\qvm$, we get
$$
\qvm (v)=\qvm[(H_++\lambda)^{-1/2}f]\le\qvm [(H_++\lambda)^{-1/2}g]\le \alpha'\|g\|^2_{L^2(\R^d)}=
$$
$$
=\alpha'\|(H_++\lambda)^{1/2}v\|^2_{L^2(\R^d)}
=\alpha'[h_A(v)+\qvp (v)+\lambda\|v\|^2_{L^2(\R^d)}].
$$
\end{proof}

The Feynman-Kac-It\^o formula (3.8) can be extended to the Hamiltonian $H$ (cf. \cite{IT2}).
\begin{proposition}\label{under}
Under assumptions (i) and (ii), for any $u\in L^2(\R^d)$ and all $t\ge 0$, we have
\begin{equation}\label{usi}
\left(e^{-tH}u\right)(x)=E_x\left[(u\circ X_t)\,e^{-S(t,X)-\int^t_0 (V\circ X_s)ds}\right],\ \ \ x\in\R^d.
\end{equation}
\end{proposition}
By using ideas from \cite{S2} and Propositions 2.13, 2.15 and \ref{under}, we are going to prove
\begin{proposition}\label{esse}
Under assumptions (i), (ii), $C^\infty_0 (\RR)$ is an essential domain for the form $h$.
\end{proposition}
\begin{proof}
Due to Hypothesis (ii) the form $h$ and the operator $H$ are well-defined.

1. Let us first suppose that $V_-=0$. We show that $D(h)\cap L^\infty_{\rm comp}(\RR)$ is an essential domain for $h$.
It is known that the range $R(e^{-H})$ is an essential domain for $h$. By Proposition \ref{under}, for any $u\in\l$
\begin{equation}\label{crig}
|e^{-H}u|\le e^{-H(0,0)}|u|,\ \ \ {\rm a.e.\ on}\ \RR
\end{equation}
the function on the right hand side being of class $L^\infty(\mathbb{R}^d)$.

Let $u\in D(h)\cap L^\infty (\RR)$, $\psi$ and $\psi_j$ as in Proposition 2.15 and $u_j:=\psi_j u,\ j\ge 1$. Then
$$u_j\in L^\infty_{\rm comp}(\RR)\cap D(\qvp),\ \ \underset{j\to\infty}{\lim}u_j=u\ {\rm in}\ \l\ \
{\rm and}\ \underset{j\to\infty}{\lim}\qvp (u_j-u)=0.$$
Let us notice that we have the equality
\begin{equation}\label{4.11}
h_A(v,w)=\left((H_A+1)^{1/2}v,(H_A+1)^{1/2}w\right)_{\l}-(v,w)_{\l}
\end{equation}
for any
$v,w\in D(h_A)=D(H^{1/2}_A)=D[(H_A+1)^{1/2}]$.
The operator $H_A+1$ is defined, by Proposition 2.8,
by the magnetic pseudo-differential operator $\Op(<\xi>)$, so by the point (b) of Proposition 2.9, $(H_A+1)^{1/2}$
is defined by an operator $\Op(q)$, where $q\in S^{1/2}(\RR)$ and
$q-<\xi>^{1/2}\in S^{-1/2}(\RR)$.
Since
$u\in D[(H_A+1)^{1/2}]$, we shall have
$\Op(q) u\in\l$; using Proposition 2.15 (b) we infer that $\Op(q)u_j$ belongs to $\l$ and
$\underset{j\to\infty}{\lim}\Op (q)u_j=\Op(q)u$
in $\l$. Since
$\Op (q)-\Op(<\xi>^{1/2})\in B[\l]$,
by Proposition 2.3, it follows that
$\Op(<\xi>^{1/2})u_j\in\l$,
so
$u_j\in\H^{1/2}_A(\RR)=D(H^{1/2}_A)=D(h_A)$.
Also using (\ref{4.11}), we get
$$
\underset{j\to\infty}{\lim}h_A (u_j-u)=0,\ \ {\rm so}\ \underset{j\to\infty}{\lim}u_j=u\ \,{\rm in}\ D(h).
$$

2. We prove that $C^\infty_0 (\RR)$ is an essential domain for $h(A;V_+)$.
Obviously $C^\infty_0 (\RR)\subset D(h)$.
Let $u\in D(h)\cap L^\infty_{\rm comp}(\RR)$ and $R_{j}u,\ j\ge 1$, defined as in Proposition 2.13.
Then
$$R_{j}u\in C^\infty_0 (\RR),\ \ \underset{j\to\infty}
{\lim}R_{j}u=u\ {\rm in}\ \l,\ \ \underset{j\to\infty}{\lim}\Op(q)R_ju=\Op(q)u\ {\rm in}\ \l,$$
where $q$ has been defined above. It follows that $\underset{j\to\infty}{\lim}h_A(R_ju-u)=0$.

On the other hand,
$$
{\rm supp} R_ju\subset\{x\in\RR\mid{\rm dist}(x,{\rm supp} u)\le 1\},\ |(R_ju)(x)-u(x)|
\le 2\|u\|_{\i},\ x\in\RR
$$
and there is a subsequence $(R_{j_k}u)_{k\ge 1}$ such that
$
(R_{j_k}u)(x)\underset{k\to\infty}{\longrightarrow}u(x)\ \ {\rm a.e.}\ x\in\RR$.
Using the Dominated Convergence Theorem we see that
$\underset{k\to\infty}{\lim}\qvp (R_{j_k}u-u)=0$,
thus
$\underset{k\to\infty}{\lim}R_{j_k}u=u\ {\rm in}\ D(h)$.

3. In order to end the proof we have to notice that $\qvm$ is relatively bounded with respect to $h(A;V_+)$ and consequently any convergent sequence from $D(h(A;V_+))$ is also convergent in $D(\qvm)$.
\end{proof}

\begin{corollary}\label{hypo}
Under hypothesis (i) and (ii), a vector $u\in D(h)$ belongs to $D(H)$ if and only if
$\Op(p)u+Vu\in\l$,
where $p(\xi):=<\xi>-1$. Moreover
$
Hu=\Op(p)u+Vu
$
for any $u\in D(H)$.
\end{corollary}

\begin{proof}
Let $u\in D(h)$. Since $\c$ is an essential domain for $h,\  u\in D(H)$ if and only if there exists $f\in\l$
such that
$$h(u,v)=(f,v)_{\l},\ \forall\,v\in\c;$$
if this is the case, then $Hu=f$.
By Proposition 4.1, $V_{\pm}u\in L^1_{\rm loc}(\RR)$ and
\begin{equation}\label{4.12}
q_{{}_{V_{\pm}}}(u,v)=<V_{\pm}u,\overline{v}>,\ \ \ \ \forall\,v\in \c,
\end{equation}
where we denoted by $<\cdot,\cdot>$ the duality bracket between $\D(\RR)$ and $\D'(\RR)$.
Let $\{u_j\}_{j\ge 1}\subset C^\infty_0 (\RR)$ such that $\underset{j\to\infty}{\lim}u_j=u$ in $D(h)$.
Then
$\underset{j\to\infty}{\lim}H_A^{1/2} u_j=H_A^{1/2} u\ \ {\rm in}\ \l$.

Using Proposition 2.2 (a), we get
\begin{equation}\label{4.13}
h_A(u,v)=\underset{j\to\infty}{\lim}(H^{1/2}_A u_j, H^{1/2}_A v)_{\l}=
\underset{j\to\infty}{\lim}(H_A u_j,v)_{\l}=
\end{equation}
$$
=\underset{j\to\infty}{\lim}<\Op (p)u_j,\overline{v}>=<\Op (p)u,\overline{v}>,\ \ \ \forall\, v\in \c.
$$
The Proposition follows immediately from the equality
\begin{equation}\label{4.14}
h(u,v)=<\Op (p)u,\overline{v}>+<Vu,\overline{v}>,\ \ \ \,v\in\c,
\end{equation}
which is a consequence of (\ref{4.12}) and (\ref{4.13}).
\end{proof}

\begin{proposition}\label{suppo}
We suppose that (i) and (ii) are true. Let $\varphi\in BC^\infty (\RR)$ such that $|M|<\infty$, where
$M:=\underset{|\alpha|=d+2}{\cup}{\rm supp}\,\partial^\alpha\varphi$.

(a) If $u\in D(H)$, then $\varphi u\in D(H)$. Moreover the commutator $[\varphi, H]$, which is well-defined on $D(H)$,
can be extended to an element of $B[\l]$.

(b) There exists a constant $C>0$, independent of $\varphi$, operators
$$S_\alpha,\ S'_\alpha\in B[\l],\ 1\le |\alpha|\le d+1,$$
independent of $\varphi$ and operators $T,\,T'\in\mathcal I_2$, such that
\begin{equation}\label{4.15}
\|T\|_{\mathcal I_2}+\|T'\|_{\mathcal I_2}\le C\underset{|\alpha|=d+2}{\rm max}
\|\partial^\alpha\varphi\|_{\l}|M|^{1/2}
\end{equation}
and
\begin{equation}\label{4.16}
[\varphi,H]=\underset{1\le |\alpha|\le d+1}{\sum}(\partial^\alpha\varphi)S_\alpha+T=
\underset{1\le |\alpha|\le d+1}{\sum}S'_\alpha (\partial^\alpha\varphi)+T'.
\end{equation}

(c) One has
\begin{equation}\label{4.17}
[(H+\lambda)^{-1},\varphi]=(H+\lambda)^{-1}[\varphi,H](H+\lambda)^{-1},\ \ \ \forall\,\lambda\in\R,
\ \lambda>-{\rm inf}\sigma(H).
\end{equation}
\end{proposition}

\begin{proof}
(a) Let $u\in D(H)$. Then $u\in D(\qvp)\cap D(h_A)$. It follows that $\varphi u\in D(\qvp)$ and
$$u\in D(H^{1/2}_A)=D\left((H_A+1)^{1/2}\right)=\H^{1/2}_A(\RR).$$
Since $\varphi\in S^0(\RR)$ and $\Op (\varphi)$ is the operator of multiplication by $\varphi$,
by Proposition 2.6,
$$\varphi u\in \H^{1/2}_A(\RR)=D\left(H^{1/2}_A\right);$$
thus $\varphi u\in D(h)$.
By Proposition 2.10 (a) if follows that
$[\Op (p),\varphi]\in B[\l]$,
$p$ being given by Corollary 4.4. Therefore, computing in $\D'(\RR)$, we get
\begin{equation}\label{4.18}
\Op (p)(\varphi u)+V(\varphi u)=\varphi [\Op(p)u+Vu]+[\Op (p),\varphi]u\in\l.
\end{equation}
From Corollary \ref{hypo} we deduce that $\varphi u\in D(H)$. In addition, the equality (\ref{4.18}) shows that
\begin{equation}\label{4.19}
[\varphi, H]=[\varphi,\Op (p)]\ \ {\rm on}\ \D(H),
\end{equation}
which implies the last statement of point (a).

(b) follows from (\ref{4.19}) and proposition 2.14 with $m=1$ and $N=d+1$.

(c) is trivial.
\end{proof}

We close this Section with a result on gauge covariance of the operator $H$.

\begin{proposition}\label{4.6}
We assume (i) and (ii). Let $A$ be a vector potential for $B$ with components in
$C^{\infty}_{\rm pol}(\RR)$ and let $\widetilde{A}=A-d\varphi$ for some real function
$\varphi\in C^{\infty}_{\rm pol}(\RR)$. We denote by $U$ the unitary operator
of multiplication by $e^{-i\varphi}$ on $\l$. Then
\begin{equation}\label{4.20}
U\,H(A;V)\,U^{-1}=H(\widetilde{A};V).
\end{equation}
\end{proposition}

\begin{proof}
We notice first that from the equality
$$
\varphi (x)-\varphi (y)=<x-y,\int^1_0(\nabla\varphi)((1-s)x+sy)ds>
$$
and from Definition 1.2, one gets the relation
\begin{equation}\label{4.21}
[e^{-i\varphi}\Op (a)(e^{i\varphi}w)](x)=\left[\mathfrak{Op}^{\widetilde A}(a)w\right](x),\ \ \ \forall\,x\in\RR
\end{equation}
for any $a\in S^m{\RR}$ and any $w\in\S (\RR)$.

Let $u\in D[H(\widetilde A;V)]$; cf. Corollary \ref{hypo}, $u\in
D[h(\widetilde A;V)]$ and
$$\mathfrak{Op}^{\widetilde A}(p)u+Vu\in\l,$$
where
$p(\xi):=<\xi>-1$. From (\ref{4.21}) we deduce that
\begin{equation}\label{4.22}
\Op (p)(e^{i\varphi}u)+V(e^{i\varphi}u)=e^{i\varphi}[\mathfrak{Op}^{\widetilde A}(p)u+Vu]\in\l.
\end{equation}
Let us show that $e^{i\varphi}u\in D[h(A;V)]$. Obviously $e^{i\varphi}u\in D(\qvp)$. We notice now that (2.2)
implies that
$w\in D(H_{\widetilde A})\ {\rm if\ and\ only\ if}\ e^{i\varphi}w\in D(H_A)$.
From (\ref{4.21}) we get $U\,H_A\,U^{-1}=H_{\widetilde A}$, so $U\,H^{1/2}_A\,U^{-1}=H^{1/2}_{\widetilde A}$
and then $U^{-1}[D(h_{\widetilde A})]=D(h_A)$. It follows that $e^{i\varphi}u\in D(h_A)$, so $e^{i\varphi}u\in
D[h(A;V)]$. Using Corollary \ref{hypo} and equality (\ref{4.22}), we deduce $U^{-1}u\in D[H(A;V)]$ as well as
(\ref{4.20}).
\end{proof}

\section{Trace estimations}\label{simptotic}

\begin{proposition}\label{unu}
Let us suppose that (i) and (ii) are verified. There exists $\mu\ge 1$, only depending on $V_-$, such that for all
$$\lambda\ge \lambda_0:=\max\{-{\rm inf}\sigma (H)+1,\mu\},\ \ \ r\ge r_0:=d+1,$$
there exists $C>0$ such that
for every bounded open subset $\Omega$ of $\RR$ we have
$\boldsymbol{1}_{\Omega}(H+\lambda)^{-r}\in\mathcal I_2$
and
\begin{equation}\label{5.1.}
             \|\boldsymbol{1}_\Omega (H+\lambda)^{-r}\|_{\mathcal I_2}\le C|\Omega|^{1/2}.
\end{equation}
We denoted by $\boldsymbol{1}_\Omega$ both the characteristic function of $\Omega$ and the associated multiplication
operator on $\l$.
\end{proposition}

\begin{proof}
 We use (3.12) and Proposition 4.2 to obtain that for any $f\in\l,\ \lambda\ge\lambda_0$ and $r\ge r_0$ one has
 \begin{equation}\label{5.2}
 |(H+\lambda)^{-r}f|\le \frac{1}{\Gamma (r)}\int^\infty_0 t^{r-1}e^{-\lambda t}e^{-tH(0;-V_-)}|f|dt.
 \end{equation}
Since $V_-\in\K_d$, by Theorem 1.5 from \cite{vC} (or Theorem 2.9 from \cite{DvC}), for any $t>0$ the operator
$e^{-tH(0,-V_-)}$ has an integral kernel satisfying: For any $\rho,\rho' >1,\ \frac{1}{\rho}+\frac{1}{\rho'}=1$,
one can choose positive constants $M, b$ such that
\begin{equation}\label{such}
0\le e^{-tH(0,V_-)}(x,y)\le Me^{bt}\underset{z\in\RR}{\sup}[p_{t/2}(z)]^{1/\rho'}[p_t (x-y)]^{1/\rho},
\ \ \forall \,t>0,\ x,y\in\RR,
\end{equation}
where $p_t$ is defined by (3.4). Using (3.4) and (3.3) it follows that there exists an absolute constant
$C>0$ such that
\begin{equation}\label{tha}
p_t(x)\le Ct e^t\left[(|x|^2+t^2)^{-\frac{d+1}{2}}+(|x|^2+t^2)^{-\frac{d+2}{4}}\right]e^{-(|x|^2+t^2)^{1/2}},
\ \ \ \forall t>0,\ x\in\RR.
\end{equation}
 We choose $\rho=4$ and $\rho'=\frac{4}{3}$ in (\ref{such}). From (\ref{tha}) it follows that for
 some $C_1>0$:
 \begin{equation}\label{5.3}
 \underset{z\in\RR}{\sup}[p_{t/2}(z)]^{3/4}\le C_1 \left(t^{-\frac{3d}{4}}+t^{-\frac{3d}{8}}\right),\ \ \ \forall t>0.
 \end{equation}
 Using (\ref{5.2}), (\ref{5.3}), (\ref{such}) and (\ref{tha}), we get the inequality
 \begin{equation}\label{5.4}
 |[(H+\lambda)^{-r}f](x)|\le\int_{\RR}L(x-y)|f(y)|dy=:(T|f|)(x),\ \ f\in\l,\ x\in\RR,
 \end{equation}
 where
 \begin{equation}\label{5.5}
 L(x):=C_2\int^\infty_0 e^{-(\lambda-b-1/4)t}\,t^{r-\frac{3}{4}}\left(t^{-\frac{3d}{4}}+t^{-\frac{3d}{8}}
 \right)\times
 \end{equation}
 $$
 \times\left[(|x|^2+t^2)^{-\frac{d+1}{8}}+(|x|^2+t^2)^{-\frac{d+2}{16}}\right]e^{-\frac{1}{4}(|x|^2+t^2)^{1/2}}dt\le
 $$
 $$
 \le C_2\left[|x|^{-\frac{d+1}{4}}+|x|^{-\frac{d+2}{8}}\right]e^{-\frac{|x|}{4}}\int^\infty_0 e^{-(\lambda-b-
 \frac{1}{4})t}\,t^{r-\frac{3}{4}}\left(t^{-\frac{3d}{4}}+t^{-\frac{3d}{8}}\right)dt,
 $$
 where $C_2$ is a positive constant. Choosing $\mu=b+\frac{1}{2}$, the assumptions insure the convergence of the last
 integral. It follows that $L\in\l,\ \forall d\ge 2.$

 The integral operator $\boldsymbol{1}_\Omega T$ is Hilbert-Schmidt, since
 \begin{equation}\label{5.6}
 \|\boldsymbol{1}_\Omega T\|_{\mathcal I_2}=\left[\int_{\RR}\int_{\RR}|\boldsymbol{1}_\Omega (x)L(x-y)|^2dxdy\right]^{1/2}=
 \|L\|_{\l}|\Omega|^{1/2}.
 \end{equation}
The conclusion of the Proposition follows from (\ref{5.4}), (\ref{5.6}) and Theorem 2.13 from [53].

\end{proof}
 \begin{corollary}\label{doi}
Under the assumptions of Proposition \ref{unu}, for any $m\ge 2r_0$ there exists $C>0$ such that for any $\Omega
\subset\RR$ open and bounded, we have $\boldsymbol{1}_\Omega (H+\lambda)^{-m}\boldsymbol{1}_\Omega \in\mathcal I_1$ and
\begin{equation}\label{5.7}
    \|\boldsymbol{1}_\Omega  (H+\lambda)^{-m}\boldsymbol{1}_\Omega \|_{\mathcal I_1}\le C|\Omega|.
\end{equation}
\end{corollary}
\begin{proof}
We choose $r\ge r_0,\ s\ge r_0,\ r+s=m $. Then
$$
 \|\boldsymbol{1}_\Omega  (H+\lambda)^{-m}\boldsymbol{1}_\Omega \|_{\mathcal I_1}\le  \|\boldsymbol{1}_\Omega  (H+\lambda)^
 {-r} \|_{\mathcal I_2} \|(H+\lambda)^{-s}\boldsymbol{1}_\Omega \|_{\mathcal I_2},
$$
and we use Proposition \ref{unu} to conclude.
\end{proof}

\begin{corollary}\label{trei}
Let $f\in L^\infty (\R),\ {\rm supp}f\subset (-\infty,a],\ a\in\R$. Under the assumptions of Proposition \ref{unu},
$\exists C>0$ such that for any $\Omega\subset\RR$ open and bounded we have $\boldsymbol{1}_\Omega f(H)\boldsymbol{1}_\Omega
 \in\mathcal I_1$ and
 \begin{equation}\label{5.8}
               \|\boldsymbol{1}_\Omega f(H)\boldsymbol{1}_\Omega \|_{\mathcal I_1}\le C|\Omega|.
 \end{equation}
\end{corollary}

\begin{proof}
We use the equality
$$
  \boldsymbol{1}_\Omega f(H)\boldsymbol{1}_\Omega=\boldsymbol{1}_\Omega(H+\lambda)^{-r}(H+\lambda)^{2r} f(H)(H+\lambda)^{-r}
  \boldsymbol{1}_\Omega ,\ \ r\ge r_0
$$
and Proposition \ref{unu}, taking into account the fact that $H$, being lower semi-bounded, satisfies
$$(H+\lambda)^{2r} f(H)\in B[\l].$$
\end{proof}

\section{The operator $H_\Omega$}\label{brasagain}

We assume (i) and (ii) for a while; let $H=H(A;V)$ be the operator constructed in Section 4. Let $\Omega$ be
an open subset of $\R^n$ and $\Omega^c$ its complement. For $n\in\N,\,n\ge 1$, we set
$H_n:=H+n1_{\Omega^c}$,
which is a self-adjoint operator on $L^2(\R^d)$ with domain $D(H_n)=D(H)$. The associated quadratic form
\begin{equation}\label{dratic}
h_n(u,v):=h(u,v)+n(\boldsymbol{1}_{\Omega^c}u,\boldsymbol{1}_{\Omega^c}v),\ \ \ \ \ u,v\in D(h_n)=D(h)
\end{equation}
is symmetric, lower semi-bounded and closed. We also have
$h\le h_n\le h_{n+1},\ \ \ \forall n\ge 1$.

We are going to identify $L^2(\Omega)$ with the closed subspace of $L^2(\R^d)$ whose elements are null on $\Omega^c$.
The operator $\boldsymbol{1}_\Omega$ will be the orthogonal projection of $L^2(\R^d)$ on $L^2(\Omega)$.

To the monotone sequence of forms $\{h_n\}_{n\ge 1}$ defined by (\ref{dratic}) one assigns the form $h_\Omega$
defined on
\begin{equation}\label{diudiu}
D(h_\Omega):=\{u\in\cap_{n\ge 1}D(h_n)\mid\sup_{n\geq 1}h_n(u,u)<\infty\}=D(h)\cap L^2(\Omega)
\end{equation}
by the equality
\begin{equation}\label{ciuciu}
h_\Omega(u,v)=\lim_{n\rightarrow\infty}h_n(u,v)=h(u,v),\ \ \ \ \ u,v\in D(h_\Omega).
\end{equation}
The form $h_\Omega$ is not densely defined but, by Theorem 4.1 from \cite{S1}, it is lower bounded and closed,
defining a unique pseudo-self-adjoint operator $H_\Omega$ on $L^2(\R^d)$;
we have $D(H_\Omega)\subset L^2(\Omega)$, $HD(H_\Omega)\subset\l$ and $H_\Omega$, considered as
an operator in $L^2(\Omega)$, is self-adjoint. In addition,
$\lim_{n\rightarrow\infty}H_n=H_\Omega$
in strong resolvent sense. We denote by $\mathcal C_H(\R)$ the set of functions
$f:[m_f,\infty)\rightarrow\R,$
where $m_f<\inf\sigma(H)$ (maybe depending on $f$),
$f$ continuous and
$\lim_{t\rightarrow\infty}f(t)=0.$
Since $\inf\sigma(H_n)$ and $\inf\sigma(H_\Omega)$ are smaller or
equal than $\inf\sigma(H)$, one can define for any $f\in\mathcal C_H(\R)$ the operators
$$f(H_n),f(H_\Omega)\in B\left[L^2(\R^d)\right].$$

The second one is defined as follows: $f(H_\Omega)|_{L^2(\Omega)}$ is the operator from $B\left[L^2(\Omega)\right]$
associated to $H_\Omega$ (seen as a self-adjoint operator in $L^2(\Omega)$) by the usual functional calculus,
while $f(H_\Omega)=0$ on
$L^2(\Omega)^\perp$. Then we have
$\lim_{n\rightarrow\infty}f(H_n)=f(H_\Omega)$
for the strong convergence in $B\left[L^2(\R^d)\right]$. We have
\begin{equation}\label{ploch}
f(H_\Omega)=\boldsymbol{1}_\Omega f(H_\Omega)=f(H_\Omega)\boldsymbol{1}_\Omega.
\end{equation}
In particular, the properties above are checked for the function
$$f(t)=(t+\lambda)^{-1},\ \ \ \ \lambda>-\inf\sigma(H),$$
defined on a neighborhood of $\sigma(H)$. Then
$$f(H)=(H+\lambda)^{-1},\ \ \ \ f(H_\Omega)=(H_\Omega+\lambda)^{-1}.$$

\begin{lemma}\label{emma}
If we assume (i) and (ii), for any $\Omega\subset\R^d$ open bounded set, the operator $H_\Omega$ has compact resolvent.
\end{lemma}

\begin{proof}
It will be enough to show that any $M\subset D(H_\Omega)$, bounded for the graph norm defined by $H_\Omega$,
is relatively compact
in $L^2(\Omega)$. The set $M$ will be bounded in $D(h_\Omega)$, thus also bounded in $L^2(\Omega)$ and $D(h)$.
Hence the set
$M_A:=(H_A+1)^{1/2}M$
is bounded in $L^2(\R^d)$ and
$M=(H_A+1)^{-1/2}M_A.$

Let
$$\chi\in C_0^\infty(\R^d),\ \ \ 0\le\chi\le 1,\ \ \ \chi=1\ {\rm in\ a\ neighborhood\ of}\ \overline\Omega.$$
By (\ref{3.11}) one has
\begin{equation}\label{trecut}
|\chi(H_A+1)^{-1/2}f|\le\chi(H_0+1)^{-1/2}|f|,\ \ \ \ \ \forall f\in L^2(\R^d).
\end{equation}
Since
$$R(H_0+1)^{-1/2}=\H^{1/2}(\R^d),$$ the operator $\chi(H_0+1)^{-1/2}$ is compact on $L^2(\R^d)$.
By Pitt's Theorem \cite{Pi} and by
(\ref{trecut}), the operator $\chi(H_A+1)^{-1/2}$ is also compact on $L^2(\R^d)$. Since $\chi M=M$,
the set $M$ is relatively compact in $L^2(\Omega)$.
 \end{proof}

 \begin{proposition}\label{gr}
We assume that (i) and (ii') are verified. For any $\lambda\ge\lambda_0,\,r\ge r_0$ ($\lambda_0$ and $r_0$ as in
Proposition \ref{unu}), there is a constant $C>0$ such that for any open subsets $U,\Omega$ of $\R^d$ such
that $U\cap\Omega$ is bounded we have $\boldsymbol{1}_U(H_\Omega+\lambda)^{-r}\in\mathcal I_2$ and the next inequality holds:
\begin{equation}\label{securi}
\|\boldsymbol{1}_U(H_\Omega+\lambda)^{-r}\|_{\mathcal I_2}\le C|U\cap\Omega|^{1/2}.
\end{equation}
 \end{proposition}

 \begin{proof}
By using the inequality (\ref{5.2}) for $H_n$ and the fact that
$$s-\lim_{n\rightarrow\infty}(H_n+\lambda)^{-r}=(H_\Omega+\lambda)^{-r},$$
one obtains that
$$
|(H_\Omega+\lambda)^{-r}f|\le\frac{1}{\Gamma(r)}\int^\infty_0 t^{r-1}e^{-\lambda t}e^{-tH(0;-V_-)}|f| dt,\ \ \ \ \
\forall f\in L^2(\R^d).
$$
The proof is completed in the same way as for Proposition \ref{unu}, since
$$
\boldsymbol{1}_U(H_\Omega+\lambda)^{-r}=\boldsymbol{1}_U\boldsymbol{1}_\Omega(H_\Omega+\lambda)^{-r}=
\boldsymbol{1}_{U\cap \Omega}(H_\Omega+\lambda)^{-r}.
$$
 \end{proof}

  \begin{corollary}\label{grr}
Under the assumptions of Proposition \ref{gr}, for any $m\geq 2r_0$, $\exists C>0$ such that for any
$\Omega\subset\R^d$ bounded and open
 one has
 $(H_\Omega+\lambda)^{-m}\in\mathcal I_1$
 and
 \begin{equation}\label{criic}
\parallel (H_\Omega+\lambda)^{-m}\parallel_{\mathcal I_1}\le C|\Omega|.
 \end{equation}
 \end{corollary}

 \begin{proof}
We use the identity
$$
(H_\Omega+\lambda)^{-m}=\boldsymbol{1}_\Omega(H_\Omega+\lambda)^{-r}(H_\Omega+\lambda)^{-s}\boldsymbol{1}_\Omega,
$$
where
$$r\ge r_0,\ \ s\ge r_0,\ \ r+s=m,$$
as well as Proposition \ref{gr} with $U=\Omega$.
 \end{proof}

  \begin{corollary}\label{grrr}
Let $f\in C_0(\R)$. Under the assumptions of Proposition \ref{gr}, there exists a constant $C>0$
such that for any $\Omega\subset\R^d$ open and bounded, we have $f(H_\Omega)\in\mathcal I_1$ and
\begin{equation}\label{crianza}
\parallel f(H_\Omega)\parallel_{\mathcal I_1}\le C|\Omega|.
\end{equation}
 \end{corollary}

 \begin{proof}
For any $g\in \mathcal C_H(\R)$, since
$s-\lim_{n\rightarrow\infty}g(H_n)=g(H_\Omega),$
one obtains for each
$\Omega\subset\R^d$ open set, the inequality
\begin{equation}\label{aferim}
\parallel g(H_\Omega)\parallel_{B[L^2(\R^d)]}\le\sup_\R|g|.
\end{equation}
We choose
$$g(t):=(t+\lambda)^mf(t),\ \ {\rm where}\ \ m\ge 2r_0,\ \lambda\ge\lambda_0,\ -\lambda\notin{\rm supp} f.$$
Then
$$g\in\mathcal C_H(\R),\ \
-\lambda\notin {\rm supp}g\ \ {\rm and}\ \ f(t)=(t+\lambda)^{-m}g(t),\ \forall t\in\R.$$
It follows that
$$f(H_\Omega)=(H_\Omega+\lambda)^{-m}g(H_\Omega),$$
so (\ref{crianza}) is a consequence of (\ref{criic}) and (\ref{aferim}).
 \end{proof}

\section{Proof of Theorem 1.1}\label{mmrr}

\begin{lemma}\label{sum}
We assume (i) and (ii'). Let $\lambda >-\inf \sigma (H),\ \Omega\subset\RR$ an open bounded set and $\varphi\in
BC^\infty (\RR),\ \varphi=1$ on $\Omega^c$. Then one has
\begin{equation}\label{7.1}
(H+\lambda)^{-1}-(H_{\Omega}+\lambda)^{-1}=\left[(H+\lambda)^{-1} -(H_\Omega+\lambda)^{-1}\right]
\left[\varphi +[H,\varphi](H_\Omega+\lambda)^{-1}\right]=
\end{equation}
$$
=\left[\varphi-(H+\lambda)^{-1} [H,\varphi]\right]\left[(H+\lambda)^{-1} -(H_\Omega+\lambda)^{-1}\right].
$$
\end{lemma}

\begin{proof}
The function $\varphi$ verifies the assumptions of Proposition 4.5, so the operator of multiplication by $\varphi$
leaves $D(H)$ invariant and $[H,\varphi]\in B[\l]$. Using (\ref{4.17}) for $H_n$ and the equality
$$[H_n,\varphi]=[H,\varphi],\ \ \ \forall\,n\ge 1,$$
where $H_n:=H+n\boldsymbol{1}_{\Omega^c}$, we deduce that
$$
  (H+\lambda)^{-1} -(H_n+\lambda)^{-1}=(H+\lambda)^{-1} n\boldsymbol{1}_{\Omega^c}\cdot\varphi(H_n+\lambda)^{-1} =
$$
$$
  =(H+\lambda)^{-1}n\boldsymbol{1}_{\Omega^c}(H_n+\lambda)^{-1}\varphi + (H+\lambda)^{-1}n\boldsymbol{1}_{\Omega^c}
  (H_n+\lambda)^{-1}[H,\varphi](H_n+\lambda)^{-1}=
$$
$$
   =\left[(H+\lambda)^{-1} -(H_n+\lambda)^{-1}\right]\left[\varphi +[H,\varphi\right](H_n+\lambda)^{-1}].
$$
The first equality in (\ref{7.1}) follows from the formula above in the limit $n\to\infty$, taking into account
 the relation
 $$s-\underset{n\to\infty}{\lim}(H_n+\lambda)^{-1}=(H_\Omega+\lambda)^{-1}$$
and the fact that the sequence  $\{(H_n+\lambda)^{-1}\}_{n\ge 1}$ is bounded in $B[\l]$.

 The second equality in (\ref{7.1}) follows in the same way.
\end{proof}

The next Proposition is basic for proving Theorem 1.1

\begin{proposition}\label{again}
We assume again (i) and (ii'). For any $\lambda\ge\lambda_0$ and $m\in\N,\ m\ge 4r_0$ ($\lambda _0$ and $r_0$ as in
Proposition 5.1), there exists $C>0$ such that for any bounded open subset $\Omega$ of $\RR$
\begin{equation}\label{7.2}
\|\boldsymbol{1}_\Omega (H+\lambda)^{-m}\boldsymbol{1}_\Omega-(H_\Omega+\lambda)^{-m}\|_{\mathcal I_1}\le C|\Omega|^{1/2}
|\widetilde\Omega|^{1/2},
\end{equation}
where $\widetilde\Omega:=\{x\in\RR\mid {\rm dist}(x,\partial\Omega)< 1\}$.
\end{proposition}

\begin{proof}
We use (6.4) and infer that
\begin{equation}\label{7.3}
  \boldsymbol{1}_\Omega (H+\lambda)^{-m}\boldsymbol{1}_\Omega - (H_\Omega+\lambda)^{-m}=\underset{0\le j\le m-1}{\sum}
 \boldsymbol{1}_\Omega (H+\lambda)^{j-m+1}\left[(H+\lambda)^{-1}-(H_\Omega+\lambda)^{-1}\right](H_\Omega+\lambda)^{-j}
  \boldsymbol{1}_\Omega.
\end{equation}
We denote by $E_j$ the general term of the sum. Let $\Omega\subset\RR$ bounded and open. By taking the convolution
of the characteristic function of
a neighborhood of $\overline{\Omega^c}$ by a function from $\c$ with the support included in a small neighborhood
of the origin, one constructs a real function $\varphi\in BC^\infty (\RR)$ such that
$0\le\varphi\le 1$, $\varphi=1\
{\rm on}\ \Omega^c$, $\varphi = 0\ {\rm on}\ \Omega\setminus\widetilde \Omega$
and such that
$\|\partial^\alpha\varphi\|_{L^2(\R^d)}\le C_\alpha,\ \ \forall\alpha\in\mathbb{N}^d,$
with $C_\alpha$ independent of $\Omega$.

We estimate first the $\Y_1$-norm of $E_j$ for $2r_0\le j\le m-1$. We use the first equality form (\ref{7.1}) and
write
$E_j=E'_j+E''_j,$
where $E_j'$ and $E''_j$ correspond to the two terms of the sum $\varphi+[H,\varphi]$.
We have
$$
  \|E'_j\|_{\Y_1}=\|\boldsymbol{1}_\Omega (H+\lambda)^{j-m+1}\left[(H+\lambda)^{-1}-(H_\Omega+\lambda)^{-1}\right]\varphi
  (H_\Omega+\lambda)^{-j}\boldsymbol{1}_\Omega\|_{\Y_1}\le
$$
$$
\le  \|\boldsymbol{1}_\Omega (H+\lambda)^{j-m+1}\left[(H+\lambda)^{-1}-(H_\Omega+\lambda)^{-1}\right]\|_{B[L^2(\Omega)]}
\|\boldsymbol{1}_U (H+\lambda)^{-j/2}\|_{\Y_2}\|(H_\Omega+\lambda)^{-j/2}\boldsymbol{1}_\Omega\|_{\Y_2},
$$
where $U:=\Omega^c\cup\widetilde\Omega$. Using (6.9) and Proposition 6.2 we get
\begin{equation}\label{7.4}
\|E'_j\|_{\Y_1}\le C_1 |\widetilde\Omega|^{1/2}|\Omega|^{1/2}
\end{equation}
for some positive constant $C_1$, independent of $\Omega$.
To estimate the $\Y_1$ norm of $E''_j$, we write it as
$$E''_j=\underset{1\le |\alpha|\le d+1}{\sum}E''_{j,\alpha}+E''_{j,0},$$
where the terms $E''_{j,\alpha}$ and $E''_{j,0}$ correspond to the decomposition of $[H,\varphi]$ in the
second of the inequalities (\ref{4.16}). Using Propositions 4.5 and 6.2 we obtain inequalities, in which
the constants are independent of $\Omega$:
$$
  \|E''_{j,\alpha}\|_{\Y_1}=\|\boldsymbol{1}_\Omega (H+\lambda)^{j-m+1}\left[(H+\lambda)^{-1}-(H_\Omega+\lambda)^{-1}\right]
  S'_\alpha (\partial^\alpha\varphi)(H_\Omega+\lambda)^{-j-1}\boldsymbol{1}_\Omega\|_{\Y_1}\le
$$
$$
  \le  \|\boldsymbol{1}_\Omega (H+\lambda)^{j-m+1}\left[(H+\lambda)^{-1}-(H_\Omega+\lambda)^{-1}\right]
  S'_\alpha (\partial^\alpha\varphi)\|_{B[\l]}\,\cdot
$$
$$
\cdot\|\boldsymbol{1}_{\widetilde\Omega}(H_\Omega+\lambda)^{-j/2}\|_{\Y_2}\,\|(H_\Omega+\lambda)^{-j/2-1}
\boldsymbol{1}_\Omega\|_{\Y_2}\le
$$
$$
\le C'|\widetilde\Omega|^{1/2}|\Omega|^{1/2},\ \ \ \ \ 1\le |\alpha|\le d+1
$$
and
$$
  \|E''_{j,0}\|_{\Y_1}=\|\boldsymbol{1}_\Omega (H+\lambda)^{j-m+1}\left[(H+\lambda)^{-1}-(H_\Omega+\lambda)^{-1}\right]T'
  (H_\Omega+\lambda)^{-j-1} \boldsymbol{1}_\Omega\|_{\Y_1}\le
$$
$$
  \le\|\boldsymbol{1}_\Omega (H+\lambda)^{j-m+1}\left[(H+\lambda)^{-1}-(H_\Omega+\lambda)^{-1}\right]\|_{B[\l]}\,\|T'\|_{\Y_2}
  \|(H_\Omega+\lambda)^{-j-1}\boldsymbol{1}_\Omega\|_{\Y_2}\le
$$
$$
  \le C''|\widetilde\Omega|^{1/2}|\Omega|^{1/2}.
$$
Thus we have
\begin{equation}\label{7.5}
\|E''_j\|_{\Y_1}\le  C_2|\widetilde\Omega|^{1/2}|\Omega|^{1/2}.
\end{equation}
Taking (\ref{7.4}) into account we get
\begin{equation}\label{7.6}
\|E_j\|_{\Y_1}\le  C|\Omega|^{1/2}|\widetilde\Omega|^{1/2},\ \ \ 2r_0\le j\le m-1,
\end{equation}
for some constant $C>0$ independent of $\Omega$.

Let us assume now that $0\le j\le 2r_0-1$; then $m-j-1\ge 2r_0$. We use now the second equality in (\ref{7.1}) to
write
$E_j=\widetilde E'_j+\widetilde E''_j$
where, as before, $\widetilde E'_j$ and $\widetilde E''_j$
correspond to the two terms in the sum $\varphi+[H,\varphi]$. The $\Y_1$-norm of $\widetilde E''_j$ is estimated
as above, writing
$$\widetilde E''_j=\underset{1\le |\alpha|\le d+1}{\sum}\widetilde E''_{j,\alpha}+\widetilde E''_{j,0},$$
where the terms $\widetilde E''_{j,\alpha}$ and $\widetilde E''_{j,0}$ correspond to the decomposition of
$[H,\varphi]$ in the first equality in (\ref{4.16}). By (6.9) and Propositions 4.5 and 5.1, we get for $\Omega$-
independent constants
$$
  \|\widetilde E''_{j,\alpha}\|_{\Y_1}= \|\boldsymbol{1}_\Omega (H+\lambda)^{j-m}(\partial^\alpha\varphi)S_\alpha
\left[(H+\lambda)^{-1}-(H_\Omega+\lambda)^{-1}\right](H_\Omega+\lambda)^{-j}\boldsymbol{1}_\Omega\|_{\Y_1}\le
$$
$$B[\l]
\le\|\boldsymbol{1}_\Omega (H+\lambda)^{\frac{j-m}{2}}\|_{\Y_2}\,\|(H+\lambda)^{\frac{j-m}{2}}
\boldsymbol{1}_{\widetilde\Omega}\|_{\Y_2}\,\times
$$
$$
\times\|(\partial^\alpha\varphi)S_\alpha \left[(H+\lambda)^{-1}-
(H_\Omega+\lambda)^{-1}\right](H_\Omega+\lambda)^{-j}\boldsymbol{1}_\Omega\|_{B[\l]}\le
$$
$$
\le C'|\Omega|^{1/2}|\widetilde\Omega|^{1/2},\ \ \ \ 1\le |\alpha|\le d+1
$$
and
$$
 \|\widetilde E''_{j,0}\|_{\Y_1}=\|\boldsymbol{1}_\Omega (H+\lambda)^{j-m}T\left[(H+\lambda)^{-1}-(H_\Omega+\lambda)^{-1}\right]
(H_\Omega+\lambda)^{-j} \boldsymbol{1}_\Omega\|_{\Y_1}\le
$$
$$
\le \|\boldsymbol{1}_\Omega (H+\lambda)^{j-m}\|_{\Y_2}\,\|T\|_{\Y_2}\|\left[(H+\lambda)^{-1}-(H_\Omega+\lambda)^{-1}\right]
(H_\Omega+\lambda)^{-j}\boldsymbol{1}_\Omega\|_{B[\l]}\le
$$
$$
\le C''|\Omega|^{1/2}|\widetilde\Omega|^{1/2}.
$$
So we have
\begin{equation}\label{7.7}
   \|\widetilde E''_j\|_{\Y_1}\le C|\Omega|^{1/2}|\widetilde\Omega|^{1/2}.
\end{equation}
To estimate the $\Y_1$-norm of $\widetilde E'_j$, we introduce another auxiliary real function
$$\psi\in\c,\ 0\le\psi\le 1,
\ {\rm supp}\,\psi\subset\Omega\cup\widetilde\Omega,\ \psi=1\ {\rm in\ a\ neighborhood\ of}\ \overline\Omega,$$
such that for any  $\alpha\in\N^d$ one has
$\|\partial^\alpha \psi\|_{L^\infty (\RR)}\le C_\alpha,$
with $C_\alpha$ independent of $\Omega$. We have
$$
  \|\tilde{E}'_j\|_{\Y_1}=\|\boldsymbol{1}_\Omega (H+\lambda)^{j-m+1}\,\varphi[(H+\lambda)^{-1}-(H_\Omega+\lambda)^{-1}]
  (H_\Omega+\lambda)^{-j}\boldsymbol{1}_\Omega\|_{\Y_1}\le
$$
$$
\le  C\|\boldsymbol{1}_\Omega (H+\lambda)^{j-m+1}\,\varphi (H+\lambda)^{-1}\psi\|_{\Y_1}+
C\|\boldsymbol{1}_\Omega (H+\lambda)^{j-m+1}\varphi\psi\|_{\Y_1},
$$
 where we used the fact that $\psi\boldsymbol{1}_\Omega=\boldsymbol{1}_\Omega$  and we denoted by $C$ various constants
 independent of $\Omega$.

 Since ${\rm supp}\,(\varphi\psi)\subset\widetilde\Omega$, using Proposition 5.1 as above, we get
 \begin{equation}\label{7.8}
    \|\boldsymbol{1}_\Omega (H+\lambda)^{j-m+1}\,\varphi \psi\|_{\Y_1}\le C|\Omega|^{1/2}|\widetilde\Omega|^{1/2}.
 \end{equation}
 For the last term that has to be estimated we use Proposition 4.5 and write
 $$
  \varphi (H+\lambda)^{-1}=(H+\lambda)^{-1}\varphi -\underset{1\le |\alpha|\le d+1}{\sum}
(H+\lambda)^{-1}(\partial^\alpha\varphi)S_\alpha (H+\lambda)^{-1}-(H+\lambda)^{-1}T(H+\lambda)^{-1}.
 $$
 One gets immediately the inequalities
 $$
  \|\boldsymbol{1}_\Omega (H+\lambda)^{j-m}\,\varphi \psi\|_{\Y_1} \le C |\Omega|^{1/2}|\widetilde\Omega|^{1/2},
 $$
 $$
   \|\boldsymbol{1}_\Omega (H+\lambda)^{j-m}\,(\partial^\alpha\varphi)S_\alpha (H+\lambda)^{-1} \psi\|_{\Y_1}
   \le C |\Omega|^{1/2}|\widetilde\Omega|^{1/2}
 $$
 and
 $$
  \|\boldsymbol{1}_\Omega (H+\lambda)^{j-m}\,T(H+\lambda)^{-1} \psi\|_{\Y_1} \le C |\Omega|^{1/2}
  |\widetilde\Omega|^{1/2},
 $$
 which gives
 \begin{equation}\label{7.9}
   \|\boldsymbol{1}_\Omega (H+\lambda)^{j-m+1}\,\varphi  (H+\lambda)^{-1}\psi\|_{\Y_1} \le C |\Omega|^{1/2}
   |\widetilde\Omega|^{1/2}.
 \end{equation}
 From (\ref{7.8}) and (\ref{7.9}) we obtain
 \begin{equation}\label{7.10}
     \|\widetilde E'_j\|_{\Y_1} \le C |\Omega|^{1/2}|\widetilde\Omega|^{1/2}
 \end{equation}
 which, together with (7.7), implies the inequality
  \begin{equation}\label{7.11}
  \|E_j\|_{\Y_1}\le C |\Omega|^{1/2}|\widetilde\Omega|^{1/2},\ \ \ \ 0\le j\le 2r_0+1.
  \end{equation}
 The relation (7.2) follows from (\ref{7.3}), (\ref{7.6}) and (\ref{7.11}).
\end{proof}

Theorem 1.1 is a direct consequence of the next Proposition:
\begin{proposition}\label{now}
Now we assume that the hypothesis (i), (ii'), (iii) and (iv) are fulfilled. Then for any $f\in C_0 (\R)$ and
$\epsilon >0$, there exists $m_0\in \N^*$ such that
\begin{equation}\label{7.12}
  |{\rm tr}[\boldsymbol{1}_\Omega f(H)\boldsymbol{1}_\Omega]-{\rm tr}f(H_\Omega)|\le\epsilon |\Omega|
\end{equation}
for any $\Omega\in\F$ with $B(0,m_0)\subset\Omega$.
\end{proposition}
 \begin{proof}
 One uses ideas of \cite{DIM} (see also \cite{I}). Let $\lambda_0$ and $r_0$ the constants from
 Proposition 5.1. We set $a:=\lambda_0+1,\ m_0:=4r_0$. It will be enough to prove (\ref{7.12}) for the real functions
 $f\in C_0(\R)$ such that ${\rm supp}\, f\subset [-a+\frac{1}{2},\infty)$. The functions
$$\left[-a+\frac{1}{2},\infty\right)\ni t\mapsto (a+t)^{m_0} f(t)\in\R$$
and
$$[0,2]\ni\tau\mapsto \tau^{-m_0}f(\tau^{-1}-a)\in\R$$
are continuous. For any $\epsilon>0$ there is a polynomial $P_\epsilon$ with real coefficients such that
$$
  \left|\tau^{-m_0} f(\tau^{-1}-a)-P_\epsilon(\tau)\right|\le\epsilon,\ \ \ \ \forall\,\tau\in [0,2].
$$
Therefore
$$
  \left|(a+t)^{m_0}f(t)-P_\epsilon\left(\frac{1}{a+t}\right)\right|\le\epsilon,\ \ \ \ \forall\,t\ge -a+\frac{1}{2}.
$$
Let
$$Q_\epsilon (t):=(a+t)^{-m_0}P_\epsilon\left(\frac{1}{a+t}\right).$$
Then in form-sense
$$
  -\epsilon (a+H)^{-m_0}\le f(H)-Q_{\epsilon}(H)\le \epsilon (a+H)^{-m_0},
$$
so
$$
  -\epsilon\boldsymbol{1}_\Omega (a+H)^{-m_0}\boldsymbol{1}_\Omega\le \boldsymbol{1}_\Omega f(H)\boldsymbol{1}_\Omega-\boldsymbol{1}_\Omega
  Q_\epsilon (H)\boldsymbol{1}_\Omega\le\epsilon \boldsymbol{1}_\Omega (a+H)^{-m_0}\boldsymbol{1}_\Omega,\ \ \Omega\in\F.
$$
Using Corollaries 5.2 and 5.3 we obtain
\begin{equation}\label{7.13}
\left|{\rm tr} [\boldsymbol{1}_\Omega f(H)\boldsymbol{1}_\Omega]-{\rm tr} [\boldsymbol{1}_\Omega Q_\epsilon (H)\boldsymbol{1}_\Omega]\right|\le\epsilon\,{\rm tr} [\boldsymbol{1}_\Omega (a+H)^{-m_0}\boldsymbol{1}_\Omega]\le C_1\,\epsilon\,|\Omega|,
\end{equation}
where $C_1$ is a constant independent on $\epsilon$ and $\Omega\in\F$.

In the same way, using Corollaries 6.3 and 6.4, one shows that for some constant $C_2$, independent on $\epsilon$
and $\Omega\in\F$, one has
\begin{equation}\label{7.14}
\left|{\rm tr}f(\Omega)-{\rm tr}Q_\epsilon (H_\Omega)\right|\le \epsilon\,{\rm tr}(a+H_\Omega)^{-m_0}\le
C_2\,\epsilon\,|\Omega|.
\end{equation}
Inequality (\ref{7.12}) follows from (\ref{7.13}), (\ref{7.14}), (\ref{7.2}) and hypothesis (iv).

 \end{proof}

\section{Proof of Theorem 1.2}\label{secspectral}

Let us suppose that (i), (ii'), (iii), (iv) and (v) are verified. Since the proof of Theorem 1.2 is very close to that of
Theorem 1.6 from \cite{I}, we shall not indicate all the details.

Let us notice first that the hypothesis (i) and (v), the proofs of Proposition 5.1 and Corollary 5.2 from \cite{I}
show that there exists a constant magnetic field
$$B^0=\frac{1}{2}\underset{1\le j,k\le d}{\sum}B^0_{jk}dx_j\wedge dx_k,\ \ \ B^0_{jk}=-B^0_{kj}\in\R$$
and a vector potential
$$A^{(p)}=\underset{1\le j\le d}{\sum}A^{(p)}_jdx_j,$$
where the components $A^{(p)}_j$ belong to $C^\infty_{\rm pol}(\RR)$ and are
$\Gamma$-periodic, such that
$B-B^0=dA^{(p)}.$
Since $B^0=dA^0 $ with
$$A^0=\underset{1\le j\le d}{\sum}A^0_j
dx_j,\ \ \ A^0_j(x)=\frac{1}{2}\underset{1\le k\le d}{\sum}B^0_{kj}x_k,$$
we have
$B=d(A^{(p)}+A^0),$
so, by Proposition 4.6, we will assume in the sequel that the vector potential defining the
magnetic field $B$ is $A:=A^{(p)}+A^0$.

For $\gamma\in\Gamma$ we define the function $\varphi_\gamma :\RR\to\R$,
$$
\varphi_\gamma (x):=\underset{1\le j\le d}{\sum}A^0_j (\gamma)x_j =\frac{1}{2}\underset{1\le j,k\le d}{\sum}
B^0_{kj}\gamma_k x_j
$$
(so $d\varphi_\gamma =A^0 (\gamma)$)
and the unitary operators of multiplication with $e^{i\varphi_\gamma}$ on $\l$ denoted by
$U_\gamma$. Let us put
$\left(L_\gamma u\right)(x):=u(x-\gamma)$
and
$T_\gamma :=U_\gamma L_\gamma.$
The operators $T_\gamma$ are the magnetic translations \cite{Z}.
\begin{lemma}\label{sup}
Let us suppose that (i), (ii') and (v) are verified. Then the operator $H=H(A;V)$ constructed in Section 4 commutes with
$T_\gamma$, i.e.
\begin{equation}\label{8.1}
HT_\gamma=T_\gamma H,\ \ \ \forall\,\gamma\in\Gamma.
\end{equation}
\end{lemma}

\begin{proof}
By Proposition \ref{4.6}, we have $H(A;V)U_\gamma =U_\gamma H(A-d\varphi_\gamma ;V)$. So we only need to show that
 \begin{equation}\label{8.2}
 L_\gamma H(A;V)=H(A-d\varphi_\gamma;V)L_\gamma,\ \ \ \ \forall\,\gamma\in\Gamma.
 \end{equation}
 Since
 $$\Gamma^A (x-\gamma,y-\gamma)=\Gamma^A (x,y)-A^0 (\gamma)=\Gamma^{A-d\varphi_\gamma}(x,y),\ x,y\in\RR,$$
 it follows that for any
 $a\in S^m (\RR),\ \Gamma$-periodic in $x$, and for any $w\in\S (\RR)$, we have
 \begin{equation}\label{8.3}
 L_\gamma \Op (a)w=\mathfrak{Op}^{A-d\varphi_\gamma}(a)(L_\gamma w).
 \end{equation}
 Let $u\in D(H(A;V))$, so $u\in D(h(A;V))$ and $\Op (p)u+Vu\in\l$ with $p(\xi):=<\xi>-1$. From (\ref{8.3}) we have
 \begin{equation}\label{8.4}
  \left[\mathfrak{Op}^{A-d\varphi_\gamma}(p)\right](L_\gamma u)+V(L_\gamma u)=L_\gamma\left[\Op (p)u+Vu\right]\in\l.
 \end{equation}
 Let us show that $L_\gamma u\in D(h(A-d\varphi_\gamma;V))$. Obviously $L_\gamma u\in D(\qvp)$.
 From (2.2) it follows that
 $$w\in D(H_A)\ {\rm if\ and\ only\ if}\ L_\gamma w\in D(H_A)=D(H_{A-d\varphi_\gamma}).$$
 From (\ref{8.3}) we deduce that
 $L_\gamma H_A L^{-1}_\gamma=H_{A-d\varphi_\gamma},$
 so
$L_\gamma H^{1/2}_A L^{-1}_\gamma=H^{1/2}_{A-d\varphi_\gamma}.$
It follows that
$L_\gamma\left[D(h_A)\right]=D(h_{A-d\varphi_\gamma}),$
so
$L_\gamma u\in D(h_{A-d\varphi_\gamma})$
and then
$$L_\gamma u\in D\left[h(A-d\varphi_\gamma;V)\right].$$
From Corollary 4.4  and equality (\ref{8.4}) we get
$L_\gamma u\in D\left[H(A-d\varphi_\gamma;V)\right]$
as well as  (\ref{8.2}).
\end{proof}

The family $\{T_\gamma\}_{\gamma\in\Gamma}$ satisfies
$$T_\alpha T_\beta = e^{-i\varphi_\beta (\alpha)}T_{\alpha+\beta},\ \ \alpha,\,\beta\in\Gamma,$$
so it doesn't form a group. However, using \cite{A} as a model (cf. also \cite{I}), one can define
a $\Gamma$-trace for a class of operators on $B(\l)$ commuting with the magnetic translations $T_\gamma$.

\begin{definition}\label{unop}
An operator $S\in B(\l)$ belongs to $\mathcal I^\Gamma_1$ if
$T_\gamma S=ST_\gamma$, $\forall\,\gamma\in\Gamma$
and if for every
$\varphi,\psi\in L^\infty_{\rm comp} (\RR)$ one has $\varphi S\psi\in\mathcal I_1$.
 \end{definition}

One can show that for all $\varphi,\varphi',\psi,\psi'\in L^\infty_{\rm comp} (\RR)$ such that
$$\underset{\gamma\in\Gamma}
{\sum}L_\gamma (\varphi\psi)=\underset{\gamma\in\Gamma}{\sum}L_\gamma (\varphi'\psi')=1,\ \ \ \forall S\in \mathcal I_1,$$
we have the equality
${\rm tr}(\varphi S\psi)={\rm tr}(\varphi' S\psi').$
This justifies

\begin{definition}\label{fie}
  Let $S\in \mathcal I_1^\Gamma$. We call $\Gamma$-{\it trace of} $S$ the expression
  $${\rm tr}_\Gamma S:={\rm tr}(\varphi S\psi),$$
  where
$\varphi,\psi\in L^\infty_{\rm comp}(\RR)$ and $\underset{\gamma\in\Gamma}{\sum} L_\gamma (\varphi\psi)=1$.
\end{definition}

One can prove (see \cite{I})

\begin{lemma}\label{atu}
Let $S=S^*\in\mathcal I^\Gamma_1$. Then $ K_S$, the integral kernel of $S$, is a locally integrable function
on $\RR\times\RR$, its restriction to the diagonal of $\RR\times\RR$ is well-defined and locally integrable and one has
\begin{equation}\label{8.5}
{\rm tr}_\Gamma S=\int_F K_S (x,x)dx,
\end{equation}
where $F$ is a fundamental domain of $\RR$ with respect to $\Gamma$.
 \end{lemma}

 From Corollary 5.3 and Lemma \ref{sup} it follows that for any $f\in C_0 (\R)$ one has $f(H)\in\mathcal I_1^\Gamma$.
 From Lemma \ref{atu} we know that the restriction to the diagonal of integral kernel $K_{f(H)}$ exists as a locally
 integrable function. Then for any $\Omega\in\F$ one has
 \begin{equation}\label{8.6}
 {\rm tr}(\boldsymbol{1}_\Omega f(H)\boldsymbol{1}_\Omega)=\int_\Omega K_{f(H)}(x,x)dx.
 \end{equation}

By the proof of Theorem 1.6 in \cite{I} we get
 \begin{equation}\label{8.7}
 \underset{\Omega\to\RR,\Omega\in\F}{\lim}\frac{1}{|\Omega|}\int_\Omega K_{f(H)}(x,x)dx=\frac{1}{|F|}\int_F K_{f(H)}
 (x,x)dx.
 \end{equation}
Then (1.9) follows from (\ref{8.6}), (\ref{8.7}) and (\ref{8.5}).

\subsubsection*{Acknowledgements} VI and RP
acknowledge partial support from the Contract no. 2-CEx06-11-18/2006. MM was
supported by the Fondecyt Grant No. 1085162 and by the N\'ucleo Cientifico ICM
P07-027-F "Mathematical Theory of Quantum and Classical Systems".

\end{document}